\documentclass[twoside, 10pt]{article}
 \RequirePackage{amsgen}
\RequirePackage{amsmath} \RequirePackage{amstext}
\RequirePackage{amsbsy} \RequirePackage{amsopn}
\RequirePackage{amsthm} \RequirePackage{amssymb}
\RequirePackage{epsfig}

\usepackage[colorlinks=true,linkcolor=black,urlcolor=blue,citecolor=blue,anchorcolor=blue,plainpages]{hyperref}
\usepackage{tikz-cd}
\usepackage{authblk}
\usepackage[mathscr]{eucal}
\usepackage{enumerate}
\theoremstyle{plain}
\newtheorem{theorem}{Theorem}[section]
\newtheorem{proposition}[theorem]{Proposition}
\newtheorem{lemma}[theorem]{Lemma}
\newtheorem{definition}[theorem]{Definition}
\newtheorem{corollary}[theorem]{Corollary}
\newtheorem{example}[theorem]{Example}
\newtheorem{remark}[theorem]{Remark}

\let\a\alpha
\let\b\beta

\let\d\delta

\let\l\lambda

\let\k\chi

\def\D{\Delta}

\def\F{{\mathbb F}}

\def\Fq{{\F_{q}}}
\def\Z{{\mathbb Z}}
\def\Zp{{\Z_{p}}}
\def\Q{{\mathbb Q}}
\def\Qp{{\Q_p}}

\def\Z{{\mathbf Z}}
\def\Zp{{\mathbf Z}_p}
\def\F{{\mathbf F}}
\def\Fq{{\mathbf F}_q}

\def\Z{{\bf Z}}
\def\Q{{\bf Q}}

\def\D{\Delta}

\def\a{\alpha}
\def\l{\lambda}

\def\b{\beta}

\begin{document}

	\title {\bf Ergodicity for uniformly differentiable modulo $p$ functions on $\Z_p$}
	
	\author{Sangtae Jeong 
		\\Department of Mathematics, Inha University, Incheon, Korea
		22212}
	\begin{NoHyper}
		\let\thefootnote\relax\footnotetext{
			The author is supported by the Basic Science Research Program through the National Research Foundation of Korea (NRF) and funded by the Ministry of Education, Science, and Technology( 2020R1A2C1A01003498).

			\\
			{${}^1$\rm E-mail}: stj@inha.ac.kr \\
			
		}
	\end{NoHyper}

	\newcommand{\Addresses}{
		\bigskip
		\footnotesize
		
		\textit{E-mail address: }\href{mailto:stj@inha.ac.kr}{\texttt{stj@inha.ac.kr}}
		}

	\bibliographystyle{alpha}
	\maketitle

\begin{abstract}
We provide an ergodicity criterion for uniformly differentiable modulo $p$ functions on $\Z_p$ in regard to the minimal level of the reduced functions by showing that ergodic conditions are explicitly found in terms of the coefficients of Mahler or van der Put for each prime $p.$
To this end, it is essential to give  an alternative, natural proof of Memi$\acute{c}$'s result regarding Mahler's coefficients estimation for uniformly differentiable modulo $p$ functions on $\Z_p.$
\end{abstract}

\tableofcontents

\section{Introduction}

\indent It is well known that every continuous $p$-adic function on $\Zp$ is uniquely expanded as an interpolation series in terms of the coefficients of Mahler and van der Put. A major task  is to characterize $p$-adic functions that have analytic  or dynamical features from the perspective of these two coefficients.

On the analytic side, the characterizations of various classes of $p$-adic functions, such as Lipschitz functions, strictly differentiable functions, and locally-analytic functions in Mahler's or van der Put's expansions has been obtained\cite{S}. Recently,  Memi$\acute{c}$ \cite{Me20b} provided necessary and sufficient conditions on Mahler's cofficients for a 1-Lipschitz function on $\Z_p$ to be uniformly differentiable modulo $p.$

On the dynamical side, a study in this direction was motivated by poineering works  of Anashin's in \cite{An1, An2}
who provided a description of measure-preservation and ergodicity of  $1$-Lipschitz functions on $\Zp$ in Mahler's coefficients. He also defined several subclasses of $1$-Lipschitz functions and characterized their ergodicity in terms of the level of the reduced functions. He then raised an open question of finding necessary and sufficient conditions for uniformly differentiable modulo $p$ functions on $\Z_p$ to be ergodic.
As a partial answer to this question, Memi$\acute{c}$ et al. \cite{MM} gave an ergodicity criterion for such functions in terms of intrinsic data associated with the given function (see Theorem \ref{Mcri}).
Along this line of research, Jeong \cite{J21b} gave necessary and sufficient conditions for ergodic $1$-Lipschitz functions of $\mathcal{B}$-class in expansions of Mahler and van der Put. As of now, the theory of $p$-adic dynamical systems has been intensively developed \cite{AKY14a, FL,DZ, CFF,  Me20a,JL17,J18,Ya} and the reader can consult with the monographs \cite{AK,Si} and the references therein for further development.

The aim of this paper is twofold, in the sense that we examine both analytical and dynamical properties of uniformly differentiable modulo $p$ functions on $\Z_p.$
First, we revisit Memi$\acute{c}$'s result regarding Mahler's coefficients of uniformly differentiable modulo $p$ functions on $\Z_p$ and provide an alternative, natural proof of her result because her arguments apply to only the case of an odd prime $p.$ The new proof reveals a clearer picture of the function in question through close interplay between two aforementioned coefficients as well as interesting properties associated with several double sums for binomial coefficients.

Second, we revisit Memi$\acute{c}$ et al.'s ergodicity criterion for uniformly differentiable modulo $p$ functions on $\Z_p$ and provide a complete description of ergodicity of such functions in terms of the minimal level of the reduced functions. Moreover, we characterize ergodicity for uniformly differentiable modulo $p$ functions on $\Z_p$ in terms of the coefficients of Mahler and van der Put for the cases where $p=2$ or 3. For any prime $p\geq 5$, we provide two methods of finding ergodic conditions for a uniformly differentiable modulo $p$, 1-Lipschitz function on $\Z_p.$

The remainder of this paper is organized as follows. Section 2 contains a brief discussion of the prerequisites, consisting of two bases of Mahler and van der Put, dynamical properties of $1$-Lipschitz functions on $\Zp,$ and basics on uniformly differentiable modulo $p$ functions on $\Z_p.$
Section 3 provides interesting calculations on Mahler's coefficients that characterize uniformly differentiable modulo $p$ functions on $\Z_p.$ Finally, Section 4 is devoted to presenting an ergodicity criterion of uniformly differentiable modulo $p$ functions on $\Z_p$ in terms of the minimal level of the reduced functions by showing that ergodic conditions of coefficients can be explicitly given for each case of $p.$

\section{Preliminaries}

\subsection{Two bases of Mahler and van der Put}

\indent Let $\Zp$ be the ring of $p$-adic integers for a prime number $p,$ $\Q_p$ be the ring of  $p$-adic numbers, and
$|\cdot|$  be the (normalized) absolute value on $\Q_p$ associated with the additive valuation $v_{p}$ on $\Q_p$  such that
$ |x|_{p} = p^{-{v}_p(x)}$. Let $C(\Zp,\Qp)$ be the space of all continuous functions from $\Zp$ to $\Qp$.
This space is then well known to be a $\Qp$-Banach space under the supremum norm $||f||_{\rm sup}$ defined
by $||f||_{\rm sup} = {\rm max} \{ |f(x)| : x \in \Zp \}$ for $f \in C(\Zp,\Qp)$.

\begin{definition}
A sequence $\{f_n\}_{n\geq0}$ in $C(\Zp,\Qp)$ is called an orthonormal basis for $C(\Zp,\Qp)$ if and only if the
following two conditions are satisfied:
\par
{\rm(1)} every $f \in C(\Zp,\Qp)$ can be expanded uniquely as $f
=\sum_{n=0}^{\infty}a_{n}f_n$, with $a_{n} \in \Qp \rightarrow 0$ as $n \rightarrow \infty$;
\par
{\rm(2)} the sup-norm of $f$ is given by $\|f\|_{\rm sup}
=\mbox{max}\{|a_{n}|\}$.
\end{definition}

Herein,  we present a brief review for the van der Put and Mahler orthonormal bases in $C(\Zp, \Qp).$ The required notation and terminology associated with the van der Put basis is first introduced. For an integer $m\geq0$ written $p$-adically as $m=m_0 +m_1p +\cdots + m_{s}p^{s}$ for $0 \leq m_{i} \leq p-1$ and $m_{s}  \not =0$, $m_{-}$ is defined as $$ m_{-}=m-m_{s}p^{s}.$$
For an integer $m\geq 0$ and $x\in \Zp$, the van der Put functions $\{\chi(m,x)\}_{m\geq 0}$ are defined by
the characteristic functions of certain balls $B_{p^{-\lfloor \log_p m\rfloor -1}}(m)$ centered around $m$ of radius
$p^{-\lfloor \log_p m\rfloor -1},$ where by convention it is assumed that $\lfloor \log_p 0\rfloor =1$. That is,
\begin{equation*}\label{putdef}
\chi(m,x) = \left \{ \begin{array}{ll}
1 &~\mbox{if}~|x-m|\leq p^{-\lfloor \log_p m\rfloor -1};\\
0 &~\mbox{otherwise.}
\end{array}\right.
\end{equation*}

From \cite{Ma2,vP}, it is shown that $\{\chi(m,x)\}_{m\geq 0}$
is an orthonormal basis for $C(\Z_p,\Q_p).$
Hence, every continuous function $f: \Zp \rightarrow \Qp$ can be expanded uniquely as
\begin{equation}\label{Vanexp}
f(x)=\sum_{m=0}^{\infty}B_{m}\chi(m,x),
\end{equation}
with $B_{m}\in \Q_p  \rightarrow 0$ as $m  \rightarrow \infty.$
Moreover, the expansion coefficients $B_{m}$ are determined by the following formula:
\begin{equation}\label{vancf}
B_{m}=\left \{ \begin{array}{ll}
f(m)-f(m_-)~&~\mbox{if}~~m\geq p;\\
f(m)~   &~\mbox{otherwise.}
\end{array}
\right.
\end{equation}

We now recall Mahler's basis of $C(\Zp,\Qp).$
It is shown in \cite{Ma1,Ma2,S} that every continuous function $f: \Zp \rightarrow \Qp$  has a Mahler expansion
\begin{equation}\label{Maexp}
 f(x)=\sum_{n=0}^{\infty}a_{n}\binom{x}{n},
\end{equation}
with $a_{n} \in \Q_p \rightarrow 0$ as $n \rightarrow \infty,$
where the binomial coefficient polynomials
$\{\binom{x}{n}\}_{n\geq 0}$  are defined by
\begin{eqnarray*}\label{mahdef}
\binom{x}{n}=\frac{x(x-1)\cdots (x-n+1)}{n!}~(n\geq 1)~{\rm and} ~\binom{x}{0}=1.
\end{eqnarray*}
It is straightforward to see from the definition that for $n\geq 1$,
\begin{eqnarray}\label{binfor}
n\binom{x}{n}=x\binom{x-1}{n-1},
\end{eqnarray}
which will be used heavily without mentioning it in most of the cases.
The expansion coefficients $\{a_{n}\}_{n\geq0}$ can be recovered by the formula
\begin{eqnarray}\label{Mcof}
a_{n}=\sum_{i=0}^{n}(-1)^{n-i} \binom{n}{i}f(i)
\end{eqnarray}

For later use, it is of interest to recall the relations between the van der Put coefficients of $f$ and the Mahler coefficients in the interpolation series of $f$ in (\ref{Maexp}). The reader can consult with \cite{Ma2} for a detailed exposition.
Let $a_n(m)$ be the coefficients in the interpolation series
\begin{equation}\label{kaimx}
 \k(m,x)=\sum_{n=0}^{\infty}a_n(m) \binom{x}{n}.
\end{equation}

By considering the generating series associated with two sequences of numbers
$ \{\k(m,n)\}_{n\geq0}$ and $\{a_n(m)\}_{n\geq 0},$
the Mahler coefficients $a_n$ are explicitly determined by the formula
\begin{equation}\label{anform}
    a_n =\sum_{m=0}^{n}a_n(m)B_m,
\end{equation}
where the coefficients $a_n(m)$ satisfy the following properties.
For $0\leq n \leq m-1,$ $a_n(m) =0 $ and
for $n\geq m$,
\begin{equation}\label{anmfor}
    a_n(m) =\sum_{\a} (-1)^{n-m-p^M\a} \binom{n}{m+p^M\a},
\end{equation}
where $\a$ extends over all rational integers $\a$ satisfying
$$ 0 \leq \a \leq p^{-M}(n-m),$$
where $M= \lfloor \log_p m\rfloor +1 .$
Furthermore, the coefficients $a_n(m)$ allow the estimate
\begin{equation}\label{vpformula}
v_p(a_n(m)) \geq \lfloor \frac{n-1}{p^M-1} \rfloor.
\end{equation}
Switching the role in (\ref{kaimx}), let
$A_m(n)$ be the coefficients in the interpolation series
$$ \binom{x}{n} =\sum_{m=0}^{\infty}A_m(n)\k(m,x).$$
Then, for any integer $m\geq0,$
\begin{equation}\label{Bmform}
    B_m =\sum_{n=0}^{m}A_m(n)a_n,
\end{equation}
where the $A_m(n)$ are determined by the formula:
\begin{equation}\label{Amnfor}
A_m(n) =\left \{ \begin{array}{ll}
\binom{m}{n} ~&~\mbox{if}~~0\leq m\leq p-1;\\
\binom{m}{n} -\binom{m_{-}}{n}~   &~\mbox{otherwise.}
\end{array}
\right.
\end{equation}

\subsection{Basics on $p$-adic dynamical systems}

We begin by recalling  $1$-Lipschitz functions on $\Zp.$
\begin{definition}\label{defpk}
A function $f: \Zp\rightarrow\Zp $ is said to be $1$-Lipschitz if for all $x,y \in \Zp,$ we have
$$|f(x)-f(y)|\leq |x-y|.$$
\end{definition}

Typical examples of $1$-Lipschitz functions on $\Zp$ include polynomials with coefficients in $\Zp$ and a family of $\mathcal{B}$-functions in \cite{AK}.

It should be noted that a $1$-Lipschitz function $f$ is continuous on $\Zp,$  i.e., $f \in C(\Zp,\Qp).$
The following statements are equivalent to the $1$-Lipschitz condition.
\par
(\rm L1) $f(x) \equiv f(y) \pmod{p^n}$ whenever $x \equiv y \pmod{p^{n}}$ for any integer $n\geq1;$
\par
(\rm L2)  $ f(x+p^{n}\Zp) \subset f(x)+ p^{n}\Zp$ for all $x\in \Zp$ and any integer $n\geq1;$
\par
(\rm L3) $|f(x+y)-f(x)| \leq |y|$ for all $x,y \in \Zp$;
\par
(\rm L4) $|\Phi_1 f(x,y):=\frac{1}{y}(f(x+y)-f(x))| \leq 1$ for all $x \in \Zp$ and all $ y \not =0 \in \Zp;$
\par
(\rm L5)   $||\Phi_1 f(x,y)||_{\rm sup} \leq 1$ for all $y \not =0 \in \Zp.$


It should be noted that (L1) implies that a $1$-Lipschitz function $f: \Zp\rightarrow \Zp $ induces a sequence of reduced functions $f_{/n}~(n\geq1)$ on
quotient rings defined by
$$f_{/n}: \Zp/ p^{n}\Zp \rightarrow \Zp/ p^{n}\Zp, ~x+p^{n}\Zp \mapsto f(x)+p^{n}\Zp.$$

The following result characterizes the 1-Lipschitz property in the expansions of van der Put and Mahler.

\begin{proposition}\label{lipness}
Let $f$ be a continuous function represented by the expansions of  van der Put and Mahler in (\ref{Vanexp}) and
 (\ref{Maexp}).
 Then $f$ is 1-Lipschitz if and only if the two coefficients satisfy the following estimates.

(i)  $|B_{m}|\leq p^{-\lfloor \log_p m\rfloor}$ for all $m\geq 0$.

(ii) $|a_{n}|\leq p^{-\lfloor \log_p n\rfloor}$ for all $n\geq 0$.
\end{proposition}
\begin{proof}
See \cite[Theorem 5]{AKY14} or \cite[Proposition 3.1]{J13} for (i)
and \cite{An2} or \cite[Theorem 3.53]{AK} for (ii).
\end{proof}

Let us briefly recall some elements of non-Archimedean dynamics on $\Zp.$ \\
A $p$-adic dynamical system on $\Zp$ is understood as a triple $(\Zp, \mu, f)$, where $f:\Zp\rightarrow \Zp$ is a measurable function and $\mu$ is the natural probability measure on $\Zp$, which is normalized such that $\mu(\Zp)=1$. The elementary $\mu$-measurable sets are the $p$-adic balls of radius $p^{-k}$; these are the sets of the form $a+ p^k\Zp$ for $a \in \Zp$ and an integer $k\geq0$. The measure of such a ball is defined as its radius, i.e., $\mu(a+ p^k\Zp)=1/p^k$.

\begin{definition} Let $(\Zp, \mu, f)$ be a $p$-adic dynamical system on $\Zp$.
A function $f:\Zp \rightarrow \Zp$ is said to be measure-preserving if $\mu(f^{-1}(S))=\mu(S)$ for each measurable subset $S \subset \Zp$. A measure-preserving function $f: \Zp \rightarrow \Zp$ is said to be ergodic if it has no proper invariant subsets, that is, if either $\mu(S)=1$ or $\mu(S)=0$ holds for any measurable subset $S \subset \Zp$ such that $f^{-1}(S) =S$.
\end{definition}
Let $\cal S$ be a finite set of $N\geq1$ elements and $f$ be a self-map on $\cal S$. Moreover, let $f^{n}$ denote the $n$-th iterate of $f$, with the convention that $f^{0}$ is the identity map on $\cal S.$  We say that
$\cal S$  forms a single cycle of $f$, or that $f$ is transitive on $\cal S$, if $\{ x_0, f(x_0) \cdots, f^{N-1}(x_0) \}=\cal S$ for any fixed initial point $x_0 \in S.$

The following equivalent statements for measure-preservation are of fundamental importance.

\begin{proposition}\label{mpeq}
Let $f: \Zp \rightarrow \Zp$  be a $1$-Lipschitz function.
The following are equivalent:

\par
{\rm(1)} $f$ is onto;
\par
{\rm(2)} $f$ is an isometry; that is, $|f(x)-f(y)|=|x-y|$ for all $x, y\in \Zp$;
\par
{\rm(3)} $f_{/n}$ is bijective  for all integers $n\geq 1$; and
\par
{\rm(4)} $f$ is measure-preserving.

\end{proposition}

\begin{proof}
See \cite[Proposition 4]{DP}.
\end{proof}

Regarding ergodicity, we have the following equivalent statements.

\begin{proposition}\label{ergeq}

Let $f:\Zp \rightarrow \Zp$  be a measure-preserving $1$-Lipschitz  function.
Then the following are equivalent:
\par
{\rm(1)} $f$ is minimal, that is, the forward orbit of $f$ at $x$ is dense in $\Zp$ for every $x$ $\in \Zp$;
\par
{\rm(2)} $f$ is ergodic;
\par
{\rm(3)} $f_{/n}$ is transitive on $\Zp/p^n\Zp$ for all integers $n\geq1$;
\par
{\rm(4)} $f$ is conjugate to the translation $t(x)=x+1$ on $\Zp;$ and
\par
{\rm(5)} $f$ is uniquely ergodic, that is, there is only one ergodic measure.
\end{proposition}

\begin{proof}
See \cite[Theorem 6]{DP}.
\end{proof}

\begin{proposition}\label{minred}
A polynomial, $f \in \Zp[x]$, is minimal if and only if $\left(\Zp/p^{\mu}\Zp, f_{/\mu}\right)$ is minimal, where $\mu= 3$ if $p$ is $2$ or $3$ and $\mu = 2$ if $p \geq 5$.
\end{proposition}

\begin{proof}
See \cite[Proposition 9]{DP}.
\end{proof}

\subsection{Uniformly differentiable modulo $p$ functions}

In this subsection, we introduce differentiability modulo $p^k$ of $p$-adic functions on $\Zp$. This notion is of high importance in applications to automata and combinatorics. The reader may refer to \cite{AK} for more details.

\begin{definition}
A function $f:\Zp \rightarrow \Zp$ is said to be uniformly differentiable modulo $p^k$ if there exist a positive integer $N$ and $\partial_{k}f(u) \in \Qp$
such that for any integer $s\geq N$ and any $h\in \Zp,$ the congruence
\begin{eqnarray}\label{diffmodpk}
f(u+p^s h) \equiv f(u) +p^sh\partial_{k}f(u) \pmod{p^{k+s}}
\end{eqnarray}
holds, where $\partial_{k}f(u)$ does not depend on $s$ and $h.$  The smallest of these $N$ is denoted by
$N_{k}(f).$
\end{definition}

From now onward, we assume that a uniformly differentiable modulo $p$ function $f:\Zp \rightarrow \Zp$ with $N_{1}(f)=1,$ satisfies the 1-Lipschitz condition, being equivalent to that $\partial_{1}f(u)$ lies in $\Z_p$ for all $u\in \Z_p.$
Such functions include polynomials with coefficients in $\Zp$, analytic functions whose Taylor expansion coefficients lie in the $p$-adic integers
and rational functions with $\Zp$-coefficients whose denominator has no zero modulo $p.$ From the 1-Lipschitz condition of $f$, it is evident that for any integer $s\geq1$ and any $h\in \Zp,$ there exists a $\partial_{1}f(u)$ in $\Zp$ such that the congruence holds:
\begin{eqnarray}\label{diffmodp1}
f(u+p^s h) \equiv f(u) +p^sh\partial_{1}f(u) \pmod{p^{s+1}}.
\end{eqnarray}

In terms of van der Put's coefficients, the crucial relations follows by taking $u=r,$ and $h=\ell$ in (\ref{diffmodp1}): for  all $ s \geq  1$, $0\leq r<p^s$, and $1 \leq \ell <p$,
\begin{eqnarray}
B_{r+ \ell p^s} &\equiv& \ell B_{r+ p^s} \pmod{p^{s+1}}; \label{ply2}\\
B_{r+ \ell p^s} &\equiv& \ell p^s \partial_{1}f(r)\pmod{p^{s+1}}. \label{ply3}
\end{eqnarray}
The periodicity of  $\partial_{1}f(u)$ from \cite[Proposition 3.32]{AK} implies the following:
\begin{eqnarray}\label{equiv2}
r \equiv \bar{r}  \pmod{p} ~{\rm with }~0 \leq \bar{r} <p~  \Rightarrow \partial_{1}f(r) \equiv \partial_{1}f(\bar{r}) \pmod{p}.
\end{eqnarray}
Through (\ref{ply2}), (\ref{ply3}), and (\ref{equiv2}), we have the following equivalence in terms of the normalized van der Put coefficients $b_{r+lp^s}$ from Proposition \ref{lipness} defined by $b_{r+lp^s}=p^{-s}B_{r+lp^s}.$
For all $s\geq 1$ and $0 \leq \bar{r} <p$ with $\bar{r} \equiv r\pmod{p},$
 we have
\begin{eqnarray}\label{equsns3}
b_{r+ p^s}\equiv \partial_{1}f(r) \equiv \partial_{1}f(\bar{r}) \equiv b_{\bar{r}+ p}\pmod{p}.
\end{eqnarray}
Equation (\ref{equsns3}) immediately implies the following equivalence:
\begin{eqnarray}\label{equivred}
b_{r+ p^s}  \not \equiv 0 \pmod{p} \Leftrightarrow b_{\bar{r}+ p}  \not \equiv 0  \pmod{p}.
\end{eqnarray}

\section{Estimation of the Mahler coefficients}

\subsection{Coefficients estimation for $p \geq 3$}

We estimate the Mahler coefficients of  uniformly differentiable modulo $p$, 1-Lipschitz functions on $\Zp.$ This result was first proved by Memi$\acute{c}$
whose proof involves a little tricky induction arguments based on the formula in \cite[Lemma 2.2]{Me20b} for Mahler's coefficients in terms of the van der Put coefficients.

We present a natural alternative proof by  resorting to close relations between two bases of Mahler and van der Put.
Its merit applies to the case of $p=2,$ where Memi$\acute{c}$'s induction arguments do not work.

The following two lemmas are essentially crucial in computations involving binomial coefficients.
\begin{lemma}\label{valpro}
Let $l$ be an integer such that $1 \leq l \leq p-1$ and let $s$ and $j$ be positive integers such that $j<lp^s.$
Then $v_{p}(\binom{lp^s}{j})=s-v_p(j).$
\end{lemma}

\begin{proof}
The proof is identical to  the case where $l=1$ in \cite[Lemma 3.2]{MM}.
For the sake of completeness, we provide a proof.
Write $j=j_{r}p^r + \cdots + j_{s-1}p^{s-1} +j_sp^s$ in $p$-adic form, where $r=v_p(j)$ .
Then $$lp^s-j=(p-j_r)p^r + \sum_{i=r+1}^{s-1}(p-1-j_i)p^i +(l-1-j_s)p^s.$$
Thus, we obtain
$$l_p(j) +l_p(lp^s-j) =(p-1)(s-r)+ l,$$
where $l_p(n)$ denotes the sum of $p$-adic digits of $n$ in $p$-adic form.
Using the Legendre formula, the $p$-adic valuation of a binomial coefficient $\binom{lp^s}{j}$ is given by
 $$v_{p}(\binom{lp^s}{j})=
 \frac{1}{p-1}( l_p(j) +l_p(lp^s-j)-l_p(lp^s))=s-r.$$
\end{proof}

\begin{lemma}\label{bipro}
Let $s$ be an integer $\geq 2$ and $\a, \b$ be non-negative integers $ < p^{s-1}$ and let $r, l$ be integers such that $1 \leq r, l <p.$ 
Then
	\begin{enumerate}
			\item[(1)]
$$ \binom{\a + p^s}{\b +lp^{s-1}}  \equiv \binom{\a}{\b}\binom{p}{l}  \pmod{p^2};$$

\item[(2)]	
\begin{equation*}
\binom{\a + rp^{s-1}+ p^s}{\b +lp^{s-1}}
\equiv \left \{ \begin{array}{ll}
			\binom{\a +rp^{s-1}}{\b +lp^{s-1}} +\sum_{j=1}^{l}\binom{\a +rp^{s-1}}{\b+(l-j)p^{s-1}} \binom{p}{j}  \pmod{p^2}  ~\mbox{if}~ l\leq r;\\
\sum_{j=l-r}^{l}\binom{\a +rp^{s-1}}{\b+(l-j)p^{s-1}} \binom{p}{j}  \pmod{p^2}  ~\mbox{if}~ l > r;
\end{array}
\right.
\end{equation*}

		\item[(3)]	
\begin{equation*}
\binom{\a + rp^{s-1}+ p^s}{\b +lp^{s-1} +p^s}
\equiv \left \{ \begin{array}{ll}
\binom{\a +rp^{s-1}}{\b +lp^{s-1}} +\sum_{j=1}^{r-l}\binom{\a +rp^{s-1}}{\b+(l+j)p^{s-1}} \binom{p}{j}  \pmod{p^2} ~&~\mbox{if}~l<r;\\
\binom{\a +rp^{s-1}}{\b +lp^{s-1}}\pmod{p^2}~   &~\mbox{if}~ l=r.
\end{array}
\right.
\end{equation*}

		\end{enumerate}
	
\end{lemma}

\begin{proof}
This follows from a direct calculation using Lemma \ref{valpro} with $l=1$. We prove only part (3) for $l<r.$
We use the Vandermonde identity:
$$ \binom{\a + rp^{s-1}+ p^s}{\b +lp^{s-1} +p^s}=\sum_{i=0}^{\b +lp^{s-1} +p^s} \binom{\a + rp^{s-1}}{i}
\binom{p^s}{\b +lp^{s-1} +p^s-i}.$$
By Lemma \ref{valpro}, we only collect indices $i\leq \a + rp^{s-1}$ such that $\b +lp^{s-1} +p^s-i=jp^{s-1}$ for $1 \leq j\leq p.$ Such $i$'s are of the form $ i=\b + (p+l-j)p^{s-1}$ with $p+l-j\leq r.$
By Lucas's congruence,
\begin{eqnarray*}
  \binom{\a + rp^{s-1}+ p^s}{\b +lp^{s-1} +p^s} &\equiv & \sum_{j=p+l-r}^{p} \binom{\a + rp^{s-1}}{ \b + (p+l-j)p^{s-1}}
\binom{p^s}{jp^{s-1}}\pmod{p^2} \\
  &=& \binom{\a +rp^{s-1}}{\b +lp^{s-1}}+ \sum_{j=p+l-r}^{p-1} \binom{\a + rp^{s-1}}{ \b + (p+l-j)p^{s-1}}
\binom{p}{j}\pmod{p^2}.
\end{eqnarray*}
By changing $p-j$ to $j$ in the second sum we have a desired result of  part (3) for $l<r.$

A detailed verification of the rest is left for the reader.
\end{proof}

In the course of the proof of our main Theorem \ref{Mcop3over}, we require Proposition \ref{ABC} below to avoid lengthy computations involving binomial coefficients.
To this end,  we define the following quantities
associated with $\{ \d^{i}(r)\}_{1\leq i \leq 3}$ in (\ref{sdel1}), (\ref{sdel2}) and (\ref{sdel3}). For an odd prime $p$ and an integer  $r$ such that $1\leq r \leq p-1$,
put
\begin{eqnarray}
  A &:=& \sum_{l=1}^{r} \sum_{j=1}^{l} (-1)^{r-l+1}l  \binom{r}{l-j} \binom{p}{j}, \label{defA} \\
  B &:=& \sum_{l=r+1}^{p-1} \sum_{j=l-r}^{l} (-1)^{r-l+1} l  \binom{r}{l-j} \binom{p}{j}, \label{defB}  \\
  C &:=& \sum_{l=1}^{r-1} \sum_{j=1}^{r-l} (-1)^{r-l} l  \binom{r}{l+j} \binom{p}{j}.
\end{eqnarray}

\begin{proposition}\label{ABC}
  Let $p$ and $r$ be the same as before. Then $A, B,$ and $C$ satisfy the following property:
  $$  A+ B+C \equiv 0 \pmod{p^2}.$$
\end{proposition}

For a proof of Proposition \ref{ABC}, we also set the two quantities defined by
\begin{eqnarray}
  P &:=& \sum_{j=1}^{p-1-r} \sum_{l=j}^{j+r} (-1)^{r-l+1} l  \binom{r}{l-j}  \binom{p}{j}, \label{defP} \\
  T &:=& \sum_{j=p-r}^{p-1} \sum_{l=j}^{p-1} (-1)^{r-l+1} l  \binom{r}{l-j} \binom{p}{j}. \label{defT}
\end{eqnarray}
\begin{lemma}\label{Pzero}
Let $p$ be an odd prime. If $1 < r \leq p-1$, then
$P \equiv 0\pmod{p^2}.$
\end{lemma}

\begin{proof}
From $P$ in (\ref{defP}), we use the relation $i=l-j$ to rewrite it as
$$ P= \sum_{j=1}^{p-1-r} \sum_{i=0}^{r} (-1)^{r-i-j+1} (i+j)  \binom{r}{i}  \binom{p}{j}.$$
The identity $\binom{p}{j} \equiv \frac{p}{j}(-1)^{j-1} \pmod{p^2}$ yields
$$ P \equiv p\sum_{j=1}^{p-1-r} \sum_{i=0}^{r} (-1)^{r-i} (i+j)  \binom{r}{i}  \frac{1}{j} \pmod{p^2}.$$
We show that the inner sum vanishes as follows: for $r>1$,
\begin{eqnarray*}
  \sum_{i=0}^{r} (-1)^{r-i} (i+j)  \binom{r}{i}  \frac{1}{j} &=& (-1)^r+ \sum_{i=1}^{r} (-1)^{r-i} (i+j)  \binom{r}{i}  \frac{1}{j} \\
  &=& (-1)^r+ \frac{r}{j}\sum_{i=1}^{r} (-1)^{r-i} \binom{r-1}{i-1}   +
  \sum_{i=1}^{r} (-1)^{r-i} \binom{r}{i} \\
  &=& \frac{r}{j}(1-1)^{r-1}+(1-1)^r =0.
 \end{eqnarray*}
Hence, this completes the proof.
\end{proof}

We are in a position to prove Proposition \ref{ABC}.
\begin{proof}
By direct computations, it is relatively easy to see that $A+B+C \equiv 0 \pmod{p^2}$ for cases where
$r=1$ or $r=p-1$. Indeed, for $r=1$, we have
$$A+B+C = -p + \sum_{l=2}^{p-1} (-1)^{l} l \left(  \binom{p}{l-1} + \binom{p}{l} \right).$$
By well-known identities  for binomial coefficients, we obtain:
\begin{eqnarray*}
  A+B+C &=& -p + \sum_{l=2}^{p-1} (-1)^{l} l \binom{p+1}{l} = -p + \sum_{l=2}^{p-1} (-1)^{l} (p+1) \binom{p}{l-1} \\
   &\equiv & -p + \sum_{l=2}^{p-1} (-1)^{l} \binom{p}{l-1} = -(1-1)^p =0  \pmod{p^2}.
\end{eqnarray*}

For $r=p-1,$  its lengthy verification is left for the reader as an exercise.

It now remains to show that the result holds for the general case where
$1< r<p-1.$ By changing the regions of the summing variables $(l,j)$ in the addition of $A$ in (\ref{defA}) and $B$ in (\ref{defB}),  it follows from Lemma \ref{Pzero} that
$$ A+ B = P+T \equiv T \pmod{p^2}.$$
Because $A+ B +C \equiv  C+ T \pmod{p^2},$ the proof amounts to showing that  $ C+ T \equiv 0\pmod{p^2}$ by explicitly calculating $C$ and $T$ separately.
Let us do this for $C$ in the first place.
By summing over the line $ l+j=k$ in $C$, it can be seen
that
$$ C:=\sum_{l=1}^{r-1} \sum_{j=1}^{r-l} (-1)^{r-l} l  \binom{r}{l+j} \binom{p}{j} =
\sum_{k=2}^{r} \sum_{j=1}^{k-1} (-1)^{r+k-j} (k-j) \binom{r}{k} \binom{p}{j}.$$
The identity $\binom{p}{j} \equiv \frac{p}{j}(-1)^{j-1} \pmod{p^2}$ allows us to split the above sum into two parts:

\begin{eqnarray*}
  C  &\equiv & p\sum_{k=2}^{r} \sum_{j=1}^{k-1} (-1)^{r+k}  \binom{r}{k}
+  p\sum_{k=2}^{r} \sum_{j=1}^{k-1} (-1)^{r+k-1} k \binom{r}{k} \frac{1}{j} \pmod{p^2}\\
   &=& p\sum_{k=2}^{r} (-1)^{r+k}(k-1)   \binom{r}{k} +  p\sum_{k=2}^{r}(-1)^{r+k-1} k \binom{r}{k}H_{k-1},
\end{eqnarray*}
where $H_{m}$ is the $m$th harmonic number  defined by
\begin{equation}\label{Hm}
 H_m=\sum_{j=1}^{m} \frac{1}{j}.
\end{equation}

As the identity $\sum_{k=2}^{r} (-1)^{r+k}(k-1)\binom{r}{k}  =(-1)^r$ holds in the first sum, we obtain
\begin{equation}\label{Cexp}
C  \equiv (-1)^r p + p\sum_{k=2}^{r}(-1)^{r+k-1} k \binom{r}{k}H_{k-1} \pmod{p^2}.
\end{equation}
Next, we calculate $T$ from (\ref{defT}) by setting $k=l-j$ to obtain
  $$T =\sum_{j=p-r}^{p-1} \sum_{k=0}^{p-1-j} (-1)^{r-j-k+1} (j+k)  \binom{r}{k} \binom{p}{j}.$$

From both $\binom{p}{j} \equiv \frac{p}{j}(-1)^{j-1} \pmod{p^2}$
and $\sum_{k=0}^{r-1} (-1)^{r-k} (r-k)\binom{r}{k}=0,$
$T$ can be simplified as a sum of the following expressions:
 \begin{eqnarray*}
   T &\equiv & p\sum_{j=p-r}^{p-1} \sum_{k=0}^{p-1-j} (-1)^{r-k} \binom{r}{k} +p\sum_{j=p-r}^{p-1} \sum_{k=0}^{p-1-j} (-1)^{r-k} \frac{k}{j}\binom{r}{k} \pmod{p^2} \\
    &=& p\sum_{k=0}^{r-1} (-1)^{r-k} (r-k)\binom{r}{k} + p\sum_{k=0}^{r-1}\sum_{j=p-r}^{p-1-k}  (-1)^{r-k} k\binom{r}{k}\frac{1}{j} \\
     &=&p\sum_{k=1}^{r-1} (-1)^{r-k} k\binom{r}{k}(H_{p-1-k} -H_{p-1-r}),
 \end{eqnarray*}

Noting that $ H_{p-1-k} -H_{p-1-r} \equiv -(H_{r}-H_k)\pmod{p}$ in (\ref{Hm}),
we use the identity $\sum_{k=1}^{r-1} (-1)^{r-k+1} k\binom{r}{k} = r$ to obtain $T$ as follows:
\begin{eqnarray}
 T &\equiv &  p\sum_{k=1}^{r-1} (-1)^{r-k+1} k\binom{r}{k}(H_r -H_k) \pmod{p^2} \nonumber  \\
   &=&  p\sum_{k=1}^{r-1} (-1)^{r-k+1} k\binom{r}{k}H_r +  p\sum_{k=1}^{r-1} (-1)^{r-k} k\binom{r}{k}H_k  \nonumber  \\
  &=&  prH_r+ p\sum_{k=1}^{r-1} (-1)^{r-k} k\binom{r}{k}H_k = p\sum_{k=1}^{r} (-1)^{r-k} k\binom{r}{k}H_k  \nonumber \\
  &=&  pr(-1)^{r-1}+ p\sum_{k=2}^{r} (-1)^{r-k} k\binom{r}{k}H_k.\label{Texp}
\end{eqnarray}
From (\ref{Cexp}) and (\ref{Texp}), using the identity $\sum_{k=2}^{r} (-1)^{r-k} \binom{r}{k} =(-1)^r(r-1),$
 we obtain
\begin{eqnarray*}
   C+T &\equiv& (-1)^r p+ pr(-1)^{r-1} +p\sum_{k=2}^{r} (-1)^{r-k} k\binom{r}{k}(H_k-H_{k-1}) \pmod{p^2} \\
   &=& (-1)^r p+ pr(-1)^{r-1} + p\sum_{k=2}^{r} (-1)^{r-k} \binom{r}{k} \\
   &=& 0.
\end{eqnarray*}
Thus, the proof is complete.
\end{proof}

We are now ready to prove the main result.
\begin{theorem}\label{Mcop3over}
For $p$ an odd prime, let $f$ be a 1-Lipschitz function on $\Zp.$
Then, $f$ is uniformly differentiable modulo $p$ with $N_1(f)=1$ if and only if
its Mahler coefficients $\{ a_n \}_{n\geq0}$  satisfy the following conditions: writing $n=n_{-}+n_sp^s,$

(1) $a_{n_{-}+n_sp^s} \equiv 0 \pmod{p^{s+1}}$ for all $s\geq1$ and $2 \leq n_s \leq p-1;$

(2) $a_{n_{-}+p^s} \equiv 0 \pmod{p^{s+1}}$ for all $s\geq2.$

\end{theorem}

\begin{proof}
By dividing the proof  into two cases for $n,$ we  use the formulas in (\ref{anmfor}) and (\ref{vpformula})  to estimate $ a_n=\sum_{m=0}^{n} a_n(m)B_m$ in (\ref{anform}). 

For the first case where $n= n_{-}+n_sp^s$ with $n_s \geq 2$ and $s\geq1,$
we claim that $ a_{n}\equiv 0 \pmod{p^{s+1}}.$
Let $m= m_{-}+m_l p^{l} $ with $ 0\leq l \leq s.$
From the formula for $a_n(m)$ in (\ref{vpformula}), its $p$-adic valuation is given by
 $$v_{p}(a_n(m)) \geq \lfloor \frac{n_sp^s+n_{-}-1}{p^{l+1}-1}\rfloor.$$
 A direct computation yields
 $$v_{p}(a_n(m)) \geq (n_s-1) + p^{s-l-1}.$$
A simple inequality, along with $p\geq 3$, gives
$$v_{p}(a_n(m)) \geq (n_s-1) +2s -2l -1.$$
Because $B_m$ are divisible by $p^l$ from Proposition \ref{lipness}, it is easy to see that unless $s=l,$ then
 $$v_{p}(a_n(m)B_m) \geq (n_s-1) +2s -l -1\geq s+1.$$
Thus, we have
$$ a_n \equiv \sum_{m=p^s}^{n} a_n(m)B_m \pmod{p^{s+1}}.$$

Via (\ref{anmfor}), the property that $B_{m_{-}+lp^s} \equiv l B_{m_{-}+p^s} \pmod{p^{s+1}}$ in (\ref{ply2}) yields
$$ a_n\equiv \sum_{l=1}^{n_s}\sum_{m_{-}=0}^{p^s-1} (-1)^{n-m} l\binom{n_{-}+n_sp^s}{m_{-}+lp^s}B_{m_{-}+p^s}\pmod{p^{s+1}}.$$
By Lucas's theorem, a simple identity for binomial coefficients yields
\begin{eqnarray*}
   a_n &\equiv& \sum_{l=1}^{n_s}\sum_{m_{-}=0}^{p^s-1} (-1)^{n-m} l
\binom{n_{-}}{m_{-}} \binom{n_s}{l} B_{m_{-}+p^s}\pmod{p^{s+1}} \\
&\equiv& \sum_{m_{-}=0}^{p^s-1}(-1)^{n_{-}-m_{-}}\binom{n_{-}}{m_{-}} B_{m_{-}+p^s} \cdot
\sum_{l=1}^{n_s} (-1)^{n_s-l} l \binom{n_s}{l} \\
&\equiv& \sum_{m_{-}=0}^{p^s-1}(-1)^{n_{-}-m_{-}}\binom{n_{-}}{m_{-}} B_{m_{-}+p^s} \cdot n_s(1-1)^{n_s-1}\pmod{p^{s+1}} \\
&=&0.
\end{eqnarray*}
Hence, the claim follows by the assumption that $n_s\geq 2.$

Now we turn to the second case where $n= n_{-}+p^s$ with $s \geq 2$.
We will show that $ a_{n}\equiv 0 \pmod{p^{s+1}}$ for this case.
We first consider a subcase where $n_{-}<p^{s-1}.$
As in the above, we calculate the $p$-adic valuation of $a_n(m)B_m$ for $0\leq m\leq n$
For $m=m_{-}+m_lp^l$ with $1\leq m_l <p$ and $0\leq l\leq s,$ we see that
if $s-l\geq 2$, then
$$v_{p}(a_n(m)B_m) \geq p^{s-l-1} +l \geq 2s-l-1=s+(s-l)-1 \geq s+1.$$
Hence, we obtain
\begin{equation}\label{secsum}
   a_n \equiv \sum_{m=p^{s-1}}^{p^s-1} a_n(m)B_m +
\sum_{m=p^{s}}^{n} a_n(m)B_m \pmod{p^{s+1}}.
\end{equation}

From (\ref{ply3}), it is easy to deduce that the second sum in Equation (\ref{secsum}) can be expressed as follows:
\begin{eqnarray}\label{secsum2}
  \sum_{m=p^{s}}^{n} a_n(m)B_m
&\equiv& \sum_{m_{-}=0}^{n_{-}} (-1)^{n_{-}-m_{-}} \binom{n_{-}}{m_{-}}B_{m_{-}+p^s} \pmod{p^{s+1}} \nonumber \\
&\equiv& p^s\sum_{m_{-}=0}^{n_{-}} (-1)^{n_{-}-m_{-}} \binom{n_{-}}{m_{-}}\partial_1f(m_{-}) \pmod{p^{s+1}}.
\end{eqnarray}

Writing $m=m_{-}+lp^{s-1},$ we estimate the first sum in Equation (\ref{secsum}):
$$\sum_{m=p^{s-1}}^{p^s-1} a_n(m)B_m
=
\sum_{l=1}^{p-1}\sum_{m_{-}=0}^{p^{s-1}-1} (-1)^{n_{-}-m_{-}+p-l}
\binom{n_{-}+p^s}{m_{-}+lp^{s-1}}
B_{m_{-}+lp^{s-1}}.  $$

Because the congruence $\binom{n_{-}+p^s}{m_{-}+lp^{s-1}}\equiv \binom{n_{-}}{m_{-}}\binom{p}{l}\pmod{p^2}$ holds from Lemma \ref{bipro}(1) for $n_{-}<p^{s-1}$, computations similar to the first case  yield
\begin{eqnarray}
\sum_{m=p^{s-1}}^{p^s-1} a_n(m)B_m
&\equiv&
p^s\sum_{m_{-}=0}^{p^{s-1}-1}
(-1)^{n_{-}-m_{-}}\binom{n_{-}}{m_{-}} \partial_1f(m_{-}) \cdot
\sum_{l=1}^{p-1} (-1)^{p-l} \binom{p-1}{l-1} \pmod{p^{s+1}} \nonumber \\
&=&
p^s\sum_{m_{-}=0}^{p^{s-1}-1}
(-1)^{n_{-}-m_{-}}\binom{n_{-}}{m_{-}} \partial_1f(m_{-}) \cdot
\left( (1-1)^{p-1} -1 \right) \nonumber \\
&\equiv&
-p^s\sum_{m_{-}=0}^{p^{s-1}-1}
(-1)^{n_{-}-m_{-}}\binom{n_{-}}{m_{-}} \partial_1f(m_{-}) \pmod{p^{s+1}}. \label{secsum3}
\end{eqnarray}
For the current case, the desired  result of $a_n$  follows then by
adding up the two sums in (\ref{secsum2})  and (\ref{secsum3}).

Second, we consider another subcase where $n_{-}=n_* +rp^{s-1}$ with $n_*< p^{s-1}$ and $1 \leq r \leq p-1.$
Writing $m=m_{-}+p^s=m_{*}+lp^{s-1}+p^s$, let us show first that the second sum in Equation (\ref{secsum}) vanishes modulo $p^{s+1}$.
We use (\ref{anmfor}) to calculate
\begin{eqnarray*}
  \sum_{m=p^{s}}^{n} a_n(m)B_m &=& \sum_{m_{-} =0}^{n_{-}} (-1)^{n_{-}-m_{-}}\binom{n_{-}+p^s}{m_{-}+p^s}B_{m_{-}+p^s}\\
   &=& \sum_{l=0}^{r}\sum_{m_{*} =0}^{p^{s-1}-1} (-1)^{n_{*}-m_{*}+(r-l)} \binom{n_{*}+rp^{s-1}+p^s}{m_{*}+lp^{s-1}+p^s}B_{m_{*}+lp^{s-1}+p^s}
\end{eqnarray*}
From Lemma \ref{bipro} (3),  Lucas's congruence, Equation (\ref{ply3}) and \cite[Proposition 3.32]{AK}  allow us to
rewrite it as
\begin{eqnarray*}
  \sum_{m=p^{s}}^{n} a_n(m)B_m & \equiv & \sum_{l=0}^{r}\sum_{m_{*} =0}^{p^{s-1}-1} (-1)^{n_{*}-m_{*}+(r-l)} \binom{n_{*}+rp^{s-1}}{m_{*}+lp^{s-1}}B_{m_{*}+lp^{s-1}+p^s}  \\
   &\equiv& p^s \sum_{m_{*} =0}^{p^{s-1}-1}(-1)^{n_{*}-m_{*}} \binom{n_{*}}{m_{*}}\partial_1f(m_{*})\cdot \sum_{l=0}^{r} (-1)^{r-l} \binom{r}{l} \pmod{p^{s+1}}\\
   &\equiv& p^s \sum_{m_{*} =0}^{p^{s-1}-1}(-1)^{n_{*}-m_{*}} \binom{n_{*}}{m_{*}}\partial_1f(m_{*})\cdot (1-1)^{r} =0 \pmod{p^{s+1}}.
\end{eqnarray*}
Hence, the result for the second sum in Equation (\ref{secsum}) follows.

It remains to show that the first sum in Equation (\ref{secsum}) vanishes modulo $p^{s+1}$
for $n:=n_{*}+rp^{s-1}+p^s$ and $m:=m_{-}+lp^{s-1}.$

Because of the formula for $a_n(m)$ in (\ref{anmfor}),
which depends on $n=n_{*}+rp^{s-1}+p^s$ and $m=m_{-}+lp^{s-1},$
the first sum in Equation (\ref{secsum})
can be explicitly expressed as follows:
\begin{eqnarray*}
\sum_{m=p^{s-1}}^{p^s-1} a_n(m)B_m
&\equiv&   \sum_{l=1}^{p-1} l \sum_{m_{-}=0}^{p^{s-1}-1} a_n(m)B_{m_{-}+p^{s-1}} \pmod{p^{s+1}} \nonumber \\
&=&
\sum_{l=1}^{r-1} l \sum_{ m_{-}=0}^{p^{s-1}-1} \left( (-1)^{n-m} \binom{n}{m}
+ (-1)^{n-m-p^s} \binom{n}{m+p^s}  \right)B_{m_{-}+p^{s-1}} \nonumber \\
&+& r \sum_{ m_{-}=0}^{n_*}\left( (-1)^{n-m} \binom{n}{m}
+ (-1)^{n-m-p^s} \binom{n}{m+p^s}  \right)B_{m_{-}+p^{s-1}} \nonumber \\
&+& r \sum_{ m_{-}=n_* +1}^{p^{s-1}-1}(-1)^{n-m} \binom{n}{m} B_{m_{-}+p^{s-1}}
+ \sum_{l=r+1}^{p-1} l \sum_{ m_{-}=0}^{p^{s-1}-1}(-1)^{n-m} \binom{n}{m} B_{m_{-}+p^{s-1}}.
\end{eqnarray*}
Rewriting it gives
\begin{eqnarray}
\sum_{m=p^{s-1}}^{p^s-1} a_n(m)B_m
&\equiv&
\sum_{l=1}^{r-1} l \sum_{ m_{-}=0}^{p^{s-1}-1} \left( (-1)^{n-m} \binom{n}{m}
+ (-1)^{n-m-p^s} \binom{n}{m+p^s}  \right)B_{m_{-}+p^{s-1}} \nonumber \\
&+& r \sum_{ m_{-}=0}^{p^{s-1}-1}\left((-1)^{n-m} \binom{n}{m_{-}+rp^{s-1}}+(-1)^{n-m-p^s} \binom{n}{m_{-}+rp^{s-1}+p^s}\right) B_{m_{-}+p^{s-1}} \nonumber\\
&+& \sum_{l=r+1}^{p-1} l \sum_{ m_{-}=0}^{p^{s-1}-1}(-1)^{n-m} \binom{n}{m_{-}+lp^{s-1}} B_{m_{-}+p^{s-1}}  \pmod{p^{s+1}} \nonumber \\
&:=& C_1 +C_2 + C_3.\label{3sums},
\end{eqnarray}
where $C_i ~(1 \leq i \leq3)$ denotes the $i$th sum of the preceding expression. 

We estimate the three sums $\{C_i\}_{1 \leq i \leq 3}$ in (\ref{3sums}) in order.
Using Lemma \ref{bipro} we first calculate $C_1$ modulo $p^{s+1}:$
\begin{eqnarray}
  C_1 &=& \sum_{l=1}^{r-1} l \sum_{ m_{-}=0}^{p^{s-1}-1} \left( (-1)^{n-m} \binom{n_{*}+rp^{s-1}+p^s}{m_{-}+lp^{s-1}}+ (-1)^{n-m-p^s} \binom{n_{*}+rp^{s-1}+p^s}{m_{-}+lp^{s-1}+p^s}  \right)B_{m_{-}+p^{s-1}} \nonumber \\
  &=&  \sum_{ m_{-}=0}^{p^{s-1}-1} B_{m_{-}+p^{s-1}}\sum_{l=1}^{r-1} l \left( (-1)^{n-m} \binom{n_{*}+rp^{s-1}+p^s}{m_{-}+lp^{s-1}}+ (-1)^{n-m-p^s} \binom{n_{*}+rp^{s-1}+p^s}{m_{-}+lp^{s-1}+p^s}  \right) \nonumber \\
  & \equiv & \sum_{ m_{-}=0}^{p^{s-1}-1} B_{m_{-}+p^{s-1}} \D_{m_{-}}^{1}(r) \pmod{p^{s+1}}, \label{Cone}
\end{eqnarray}
 where $$\D_{m_{-}}^{1}(r):=\sum_{l=1}^{r}(-1)^{n_{*}-m_{-} +r-l+1} l\left(\sum_{j=1}^{l} \binom{n_{*}+rp^{s-1}}{m_{-}+(l-j)p^{s-1}}\binom{p}{j}
-\sum_{j=1}^{r-l}\binom{n_{*}+rp^{s-1}}{m_{-}+(l+j)p^{s-1}} \binom{p}{j}\right).$$

  By Lucas's theorem, we
  have
 \begin{eqnarray}
   \D_{m_{-}}^{1}(r) & \equiv & \sum_{l=1}^{r-1}(-1)^{n_{*}-m_{-} +r-l+1} l\left(\sum_{j=1}^{l} \binom{n_{*}}{m_{-}}\binom{r}{l-j}\binom{p}{j}
-\sum_{j=1}^{r-l}\binom{n_{*}}{m_{-}}\binom{r}{l+j}\binom{p}{j} \right)\nonumber \\
    &=& (-1)^{n_{*}-m_{-}}\binom{n_{*}}{m_{-}}\delta^{1}(r)\pmod{p^2}, \label{del1}
\end{eqnarray}
where $\delta^{1}(r)$ is given by

\begin{equation}\label{sdel1}
 \delta^{1}(r):=\sum_{l=1}^{r-1}(-1)^{r-l+1} l\left(\sum_{j=1}^{l} \binom{r}{l-j}\binom{p}{j}
-\sum_{j=1}^{r-l}\binom{r}{l+j}\binom{p}{j} \right),
\end{equation}
and  it is crucial to observe that it does not depend on $m_{-}.$
Hence, substituting (\ref{sdel1}) into $C_1$ in (\ref{Cone}) via (\ref{del1}) yields

 \begin{eqnarray}\label{C1final}
C_{1} =\sum_{ m_{-}=0}^{p^{s-1}-1}(-1)^{n_{*}-m_{-}}\binom{n_{*}}{m_{-}} B_{m_{-}+p^{s-1}} \delta^{1}(r) \pmod{p^{s+1}}.
\end{eqnarray}

Similarly, using Lemma \ref{bipro} and Lucas's congruence,  we obtain

\begin{eqnarray}\label{C2final}
 C_{2} &=&
 r \sum_{ m_{-}=0}^{p^{s-1}-1}\left( (-1)^{n-m} \binom{n}{m_{-}+rp^{s-1}}  +(-1)^{n-m-p^s} \binom{n}{m_{-}+rp^{s-1}+p^s}\right) B_{m_{-}+p^{s-1}} \nonumber\\
 &=& \sum_{ m_{-}=0}^{p^{s-1}-1}(-1)^{n_{*}-m_{-}}\binom{n_{*}}{m_{-}} B_{m_{-}+p^{s-1}} \delta^{2}(r) \pmod{p^{s+1}},
\end{eqnarray}
where
\begin{equation}\label{sdel2}
\delta^{2}(r):=
-r\sum_{j=1}^{r} \binom{r}{r-j}\binom{p}{j}.
\end{equation}

By similar computations,
we have
\begin{eqnarray}\label{C3final}
C_{3} &=& \sum_{l=r+1}^{p-1} l \sum_{ m_{-}=0}^{p^{s-1}-1}(-1)^{n-m} \binom{n}{m_{-}+lp^{s-1}} B_{m_{-}+p^{s-1}} \nonumber \\
 &=& \sum_{ m_{-}=0}^{p^{s-1}-1}(-1)^{n_{*}-m_{-}}\binom{n_{*}}{m_{-}} B_{m_{-}+p^{s-1}} \delta^{3}(r) \pmod{p^{s+1}},
\end{eqnarray}
where
\begin{equation}\label{sdel3}
\delta^{3}(r):=\sum_{l=r+1}^{p-1}(-1)^{r-l+1} l \sum_{j=l-r}^{l} \binom{r}{l-j}\binom{p}{j}.
\end{equation}
On adding up the three sums, $\{C_i\}_{1 \leq i \leq 3}$ in (\ref{C1final}), (\ref{C2final}), and (\ref{C3final}), the first sum in Equation (\ref{secsum}) is equal to:
$$ \sum_{m=p^{s-1}}^{p^s-1} a_n(m)B_m =\sum_{i=1}^{3}C_i= \sum_{ m_{-}=0}^{p^{s-1}-1}(-1)^{n_{*}-m_{-}}\binom{n_{*}}{m_{-}} B_{m_{-}+p^{s-1}} \cdot \sum_{i=1}^{3} \delta^{i}(r) \pmod{p^{s+1}}.$$
As $ \sum_{i=1}^{3} \delta^{i}(r)$ is simply $A+B+C,$ which vanishes modulo $p^2$ from Proposition \ref{ABC}, the result follows immediately. This completes a proof of necessity.

For the converse, it suffices by \cite[Lemma 2.3]{Me20b} to show that $B_{t+lp^s} \equiv lp^{s-1}B_{t_0+p}\pmod{p^{s+1}}$ for all $s\geq1 $, $t<p^s$ and $1\leq l \leq p-1.$
For $s=1$, we use the formulas in (\ref{Bmform}) and (\ref{Amnfor}) to calculate
\begin{eqnarray}
  B_{r+p} &=& \sum_{n=0}^{r+p}A_{r+p}(n)a_n = \sum_{n=0}^{r+p}(\binom{r+p}{n}-\binom{r}{n})  a_n  \nonumber \\
   &=& \sum_{n=1}^{r+p}(\sum_{j=1}^{n} \binom{r}{n-j}\binom{p}{j})  a_n =\sum_{j=1}^{r+p}\binom{p}{j}\sum_{n=j}^{r+p} \binom{r}{n-j}a_n \nonumber \\
   & \equiv & \sum_{j=1}^{p}\binom{p}{j}\sum_{n=j}^{r+p} \binom{r}{n-j}a_n \pmod{p^2}, \label{Brp}
\end{eqnarray}
because $a_n$ is divisible by $p$ for $n\geq p.$
We do the same for $m=r+lp^s$ with $s\geq2$ and $1 \leq l\leq p-1$. We begin by calculating
\begin{eqnarray*}
  B_{r+lp^s} &=& \sum_{n=0}^{r+lp^s}A_{r+lp^s}(n)a_n =\sum_{n=0}^{r+lp^s}\left(\binom{r+lp^s}{n}-\binom{r}{n}\right)  a_n \\
   &=& \sum_{n=1}^{r+lp^s}(\sum_{j=1}^{n} \binom{r}{n-j}\binom{lp^s}{j})  a_n
\end{eqnarray*}
From Lemma \ref{valpro} and the assumptions on $a_n,$ it is easily seen that if $n\geq 2p$, then
$$ v_p(\binom{lp^s}{j}a_n)  \geq s-v_p(j)+ \lfloor {\rm log}_{p} n \rfloor +1 \geq s+1.$$
Hence, we obtain
$$ B_{r+lp^s} \equiv \sum_{n=1}^{2p-1}(\sum_{j=1}^{n} \binom{r}{n-j}\binom{lp^s}{j})  a_n \pmod{p^{s+1}}.$$
$$ \equiv \sum_{j=1}^{2p-1}\sum_{n=j}^{2p-1} \binom{r}{n-j} \binom{lp^s}{j}  a_n \pmod{p^{s+1}}.$$
If $j>p$, then $\binom{lp^s}{j}\equiv l\binom{p^s}{j} \pmod{p^{s+1}},$ and $a_n$ is divisible by $p$,
we have
\begin{equation}\label{Brlps}
B_{r+lp^s}  \equiv l \sum_{j=1}^{p} \sum_{n=j}^{2p-1} \binom{r}{n-j} \binom{p^s}{j} a_n \pmod{p^{s+1}}.
\end{equation}
Here, as $\binom{p^s}{j}a_n  \equiv 0 \pmod{p^{s}}$ we observe from the Lucas's theorem that $r<p^s$ in (\ref{Brlps}) can be replaced by $0\leq r_0<p$ such that $r \equiv r_0 \pmod{p}.$
This observation yields
\begin{equation*}
B_{r+lp^s}  \equiv l \sum_{j=1}^{p}\sum_{n=j}^{2p-1} \binom{r_0}{n-j}  \binom{p^s}{j} a_n \pmod{p^{s+1}}.
\end{equation*}
Because of the congruence $\binom{p^s}{j} \equiv p^{s-1}\binom{p}{j} \pmod{p^s}$ for $1 \leq j \leq p,$
we have
\begin{equation*}
B_{r+lp^s}  \equiv l p^{s-1}\sum_{j=1}^{p} \binom{p}{j} \sum_{n=j}^{2p-1} \binom{r_0}{n-j}   a_n \pmod{p^{s+1}}.
\end{equation*}
The claim follows then by substituting (\ref{Brp}) into the right-hand side of the preceding congruence.
\end{proof}

\subsection{Coefficients estimation for $p = 2$}

As is the case with odd primes, we estimate Mahler's coefficients of uniformly differentiable modulo $2$ functions on $\Z_2.$
We first have the following analogue of Lemma \ref{bipro}.
\begin{lemma}\label{bip2}
Let $s$ be an integer $\geq 2$ and $\a, \b$ be non-negative integers $ < 2^{s-1}.$
Then
	\begin{enumerate}
			\item[(1)]$ \binom{\a +2^s}{\b +2^{s-1}}  \equiv 2 \binom{\a}{\b } \pmod{2^2};$
			
		\item[(2)]$ \binom{\a +2^s}{\b +2^{s}}  \equiv \binom{\a}{\b } \pmod{2^2};$

			\item[(3)] $\binom{\a + 2^{s-1}+2^s}{\b +2^{s-1}}  \equiv
\binom{\a + 2^{s-1}}{\b +2^{s-1}} +2\binom{\a + 2^{s-1}}{\b } \pmod{2^2};$

\item[(4)] $\binom{\a + 2^{s-1}+2^s}{\b + 2^s}  \equiv \binom{\a + 2^{s-1}}{\b} + 2\binom{\a + 2^{s-1}}{\b +2^{s-1}}\pmod{2^2};$

			\item[(5)] $\binom{\a + 2^{s-1}+2^s}{\b +2^{s-1}+2^s}  \equiv
\binom{\a + 2^{s-1}}{\b +2^{s-1}} \pmod{2^2}.$
	
		\end{enumerate}
	
\end{lemma}

\begin{proof}
This follows from a direct calculation using Lemma \ref{valpro} with $l=1.$
\end{proof}

\begin{theorem}\label{Mcop2}
Let $f$ be a 1-Lipschitz function on $\Z_2.$
Then, $f$ is uniformly differentiable modulo $2$ with $N_1(f)=1$ if and only if
its Mahler coefficients $\{ a_n \}_{n\geq0}$  satisfy the following conditions:

For all $n=n_{-}+2^s$ with $s\geq2,$
$a_{n_{-}+2^s} \equiv 0 \pmod{2^{s+1}}.$



\end{theorem}

\begin{proof}
As in odd primes, we resort to the formula in (\ref{anmfor}) to estimate $ a_n(f)=\sum_{m=0}^{n} a_n(m)B_m$ in (\ref{anform}).
Writing $n= n_{-}+2^s$ for $s\geq2$ and $m= m_{-}+2^{l} $ with $ 0\leq l \leq s,$ from the formula for $a_n(m)$ in (\ref{vpformula}), its $2$-adic valuation is given by
 $$v_{2}(a_n(m)) \geq \lfloor \frac{n-1}{2^{l+1}-1}\rfloor.$$
 A direct computation yields
 $$v_{2}(a_n(m)) \geq 2^{s-l-1} +1.$$
If $s-l\geq 2,$ then by a simple inequality, we have
$$v_{2}(a_n(m)) \geq s -l +1.$$
The property that  $B_m$ are divisible by $2^l$ from Proposition \ref{lipness}  gives
 $$v_{2}(a_n(m)B_m) \geq s+1.$$
Thus, we have
\begin{equation}\label{anp2}
  a_n \equiv \sum_{m=2^{s-1}}^{2^s-1} a_n(m)B_m + \sum_{m=2^s}^{n} a_n(m)B_m \pmod{2^{s+1}}.
\end{equation}
We first deal with the case where $n_{-}<2^{s-1}.$
From the formula for $a_n(m),$
we obtain
$$ a_n \equiv \sum_{m_{-}=0}^{2^{s-1}-1} (-1)^{n_{-}-m_{-}}\binom{n_{-}+2^s}{m_{-}+2^{s-1}} B_{m_{-}+2^{s-1}} + \sum_{m_{-}=0}^{2^s -1} (-1)^{n_{-}-m_{-}}\binom{n_{-}+2^s}{m_{-}+2^{s}} B_{m_{-}+2^{s}} \pmod{2^{s+1}}.$$
By Lemma \ref{bip2} (1) and (2), it follows that
$$ a_n \equiv \sum_{m_{-}=0}^{2^{s-1}-1} (-1)^{n_{-}-m_{-}}\binom{n_{-}}{m_{-}} 2B_{m_{-}+2^{s-1}} + \sum_{m_{-}=0}^{2^{s-1}-1} (-1)^{n_{-}-m_{-}} \binom{n_{-}}{m_{-}} B_{m_{-}+2^{s}} \pmod{2^{s+1}}.$$
The congruence that $2B_{m_{-}+2^{s-1}} \equiv 2^s \partial_1f(m_{-})\equiv  B_{m_{-}+2^{s}}$ in (\ref{ply2}) and (\ref{ply3}) yields
$$ a_n \equiv 2^{s+1} \sum_{m_{-}=0}^{2^{s-1}-1} (-1)^{n_{-}-m_{-}}\binom{n_{-}}{m_{-}} \partial_1f(m_{-}) \equiv  0\pmod{2^{s+1}},$$
completing the verification of the first case.

Now, we turn to the case where $n_{-}:=n_* +2^{s-1}$, where $n_*<2^{s-1}.$
From the formula for $a_n(m),$ it can be seen from (\ref{anp2}) that

\begin{eqnarray*}
a_n &=& \sum_{0 \leq m_{-}\leq {n_*}}\left( (-1)^{n_{-}-m_{-}} \binom{n}{m}
+ (-1)^{n_{-}-m_{-}} \binom{n}{m+2^s}  \right)B_{m_{-}+2^{s-1}} \nonumber \\
&+& \sum_{m_{-}> {n_*}}(-1)^{n_{-}-m_{-}} \binom{n}{m} B_{m_{-}+2^{s-1}}
+ \sum_{ m_{-}=0}^{2^{s-1}-1}(-1)^{n_{-}-m_{-}} \binom{n}{m} B_{m_{-}+2^{s}}. \nonumber\\
&=& \sum_{m_{-}=0}^{2^{s-1}-1} (-1)^{n_{-}-m_{-}} \binom{n_{-}+2^s}{m_{-}+2^{s-1}}B_{m_{-}+2^{s-1}}
+ \sum_{ m_{-}=0}^{2^{s-1}-1}(-1)^{n_{-}-m_{-}}\binom{n_{*}+2^{s-1}+2^s} {m_{-}+2^{s-1}+2^{s}}B_{m_{-}+2^{s-1}}   \nonumber \\
&+& \sum_{m_{-}=0}^{2^{s}-1}(-1)^{n_{-}-m_{-}} \binom{n_{-}+2^s} {m_{-}+2^{s}}B_{m_{-}+2^{s}} := A + B+C. \nonumber
\end{eqnarray*}
By Lemma \ref{bip2} (3), it follows that
\begin{eqnarray}\label{EqA}
A &=& \sum_{m_{-}=0}^{2^{s-1}-1} (-1)^{n_{-}-m_{-}} \left( \binom{n_{*}+2^{s-1}}{m_{-}+2^{s-1}}
+ 2 \binom{n_{*}+2^{s-1}} {m_{-}} \right) B_{m_{-}+2^{s-1}} \nonumber \\
 & =&  \sum_{m_{-}=0}^{2^{s-1}-1} (-1)^{n_{-}-m_{-}} \binom{n_{*}+2^{s-1}}{m_{-}+2^{s-1}}B_{m_{-}+2^{s-1}} + \sum_{m_{-}=0}^{2^{s-1}-1}(-1)^{n_{-}-m_{-}}2\binom{n_{*}+2^{s-1}}{ m_{-}} B_{m_{-}+2^{s-1}} \nonumber\\
 &:=& A_1+A_2.
\end{eqnarray}

Similarly, by Lemma \ref{bip2} (5),
\begin{eqnarray}\label{EqB}
B= \sum_{ m_{-}=0}^{2^{s-1}-1}(-1)^{n_{-}-m_{-}}
\binom{n_{*}+2^{s-1}}{m_{-}+2^{s-1}}B_{m_{-}+2^{s-1}}.
\end{eqnarray}
From  Lemma \ref{bip2} (4), it follows that 
\begin{eqnarray}\label{EqC}
    C &=&\sum_{m_{-}=0}^{2^{s-1}-1}(-1)^{n_{-}-m_{-}} \binom{n_{-}+2^s} {m_{-}+2^{s}}B_{m_{-}+2^{s}} + \sum_{m_{*}=0}^{n_*+2^{s-1}} (-1)^{n_{*}-m_{*}} \binom{n_{*}+2^{s-1}+2^s} {m_{*}+2^{s-1}+2^{s}}B_{m_{*}+2^{s-1}+2^{s}} \nonumber \\
    &=& \sum_{m_{-}=0}^{2^{s-1}-1}(-1)^{n_{-}-m_{-}}
\left(\binom{n_{*}+2^{s-1}} {m_{-}}+ 2\binom{n_{*}+2^{s-1}} {m_{-}+2^{s-1}} \right)B_{m_{-}+2^{s}} \nonumber\\
&+&\sum_{m_{-}=0}^{2^{s-1}-1}(-1)^{n_{-}-m_{-}} \binom{n_{*}+2^{s-1}} {m_{*}+2^{s-1}} B_{m_{-}+2^{s}} \nonumber\\
&\equiv& \sum_{m_{-}=0}^{2^{s-1}-1}(-1)^{n_{-}-m_{-}}
\binom{n_{*}+2^{s-1}} {m_{-}} B_{m_{-}+2^{s}} +
\sum_{m_{*}=0}^{2^{s-1}-1}(-1)^{n_{*}-m_{*}}
\binom{n_{*}+2^{s-1}} {m_{*}+2^{s-1}} B_{m_{-}+2^{s}} \nonumber \\
&:=& C_1+C_2.
\end{eqnarray}

If we add up the three sums $A, B,$ and $C$ in (\ref{EqA}),  (\ref{EqB}) and  (\ref{EqC}), then
we have
\begin{eqnarray*}
a_n= A+B+C \equiv 2A_1 +2A_2 +C_2 \pmod{2^{s+1}}
\end{eqnarray*}
because  it follows from Lucas's congruence that
$A_1 \equiv B \pmod{2^{s+1}}$ and $A_2 \equiv C_1\pmod{2^{s+1}}.$
Because $2A_2 \equiv 0 \pmod{2^{s+1}},$ and $2A_1 \equiv C_2 \pmod{2^{s+1}},$
it follows that
$$ a_n= A+B+C \equiv 2A_1 +C_2 \equiv 2C_2 \equiv 0 \pmod{2^{s+1}}.$$
Thus, this completes the proof of the ``only if" part.

For the ``if" part, it suffices by \cite[Lemma 2.3]{Me20b} to show that $B_{t+2^s} \equiv 2^{s-1}B_{t_0+2}\pmod{2^{s+1}}$ for all $s\geq1 $ and $t<2^s.$
To this end, we claim that for all $s\geq1 $ and $t<2^s,$
\begin{equation}\label{Bmfor}
 B_{t+2^s} \equiv 2^{s}(a_1 + \frac{a_2}{2} +t_0\frac{a_3}{2}) \pmod{2^{s+1}}.
\end{equation}
Therefore, the claim implies the assertion, as desired.
In fact,
substituting the congruence in (\ref{Bmfor}) with $s=1$ into (\ref{Bmfor}) again yields the assertion.
 By direct computations using (\ref{vancf}) and (\ref{Mcof}),
 the congruences in (\ref{Bmfor}) hold clearly for $s=1.$
 For all $s\geq2$ and $t<2^s$, we use the formulas in (\ref{Bmform}) and (\ref{Amnfor}) to estimate $B_{t+2^s}=\sum_{n=0}^{t+2^s}A_{t+2^s}(n)a_n.$
  As $$A_{t+2^s}=\binom{t+2^s}{n}-\binom{t}{n} =\sum_{j=1}^{n}\binom{t}{n-j}\binom{2^s}{j},$$
 from Lemma \ref{bipro} and the assumptions on $a_n$,
 we obtain, for $1\leq j \leq n,$  $$v_2(A_{t+2^s}(n)a_n)\geq s-v_2(j)+ \lfloor {\rm log}_{2} n \rfloor +1 \geq s+1.$$
Hence, we have
\begin{equation}\label{Bnp2}
  B_{t+2^s}\equiv \sum_{n=0}^{3}A_{t+2^s}(n)a_n \pmod{2^{s+1}}.
\end{equation}

Direct computation gives
\begin{eqnarray*}
  A_{t+2^s}(0) &=0,&~~A_{t+2^s}(1)\equiv 2^s \pmod{2^{s+1}}, \\
   A_{t+2^s}(2)&\equiv& 2^{s-1} \pmod{2^{s}},~~A_{t+2^s}(3)\equiv t_02^{s-1} \pmod{2^{s}}.
\end{eqnarray*}
Substituting them into (\ref{Bnp2}) gives the claim, hence completing the proof.
\end{proof}

\section{ Ergodicity criterion for uniformly differentiable modulo $p$ functions }

\subsection{Measure-preservation criterion}
Regarding ergodicty, it is essential to  examine a measure-preservation criterion for a uniformly differentiable modulo $p$, 1-Lipschitz function on $\Zp$  in terms of the coefficients of Mahler and van der Put.
To this end, we recall a measure-preservation criterion due to  Khrennikov and Yurova for 1-Lipschitz functions on $\Zp$:

From now on, we write  a 1-Lipschitz function on $\Z_p$ as
\begin{eqnarray} \label{vanexp0}
f(x) &=&\sum_{m =0}^{\infty} B_m \chi(m, x) =\sum_{m = 0}^{\infty} p^{\lfloor {\rm log}_{p}m\rfloor}b_m\chi(m, x)~~~(b_m \in \Z_p).
\end{eqnarray}

\begin{theorem}\cite[Theorem 2.1]{KY}\label{MPvan}
Let $f:\Zp \rightarrow \Zp$ be a 1-Lipschitz function in van der Put's expansion represented by (\ref{vanexp0}). Then, $f$ is measure-preserving if and only if the following conditions are satisfied:
\begin{itemize}
\item[] {\rm (i)} $\{b_0, b_1, \cdots, b_{p-1}\}$ is a complete set of all distinct residues modulo $p$;
\item[] {\rm (ii)} for any integer $s\geq 1$, $0 \leq k<p^s$,  $\{b_{k+\ell p^s}\}_{1\leq \ell \leq p-1}$ is a complete set of all distinct nonzero residues modulo $p$.
\end{itemize}
\end{theorem}
\begin{proof}
  See \cite[Proposition 3.2]{J21b} for an alternative proof.
\end{proof}

By this criterion, we deduce the following simple conditions for a measure-measure-preserving, uniformly differentiable modulo $2$, 1-Lipschitz function on $\Z_2.$

\begin{proposition}\label{MPudmp2}
Let $f$ be a uniformly differentiable modulo $2$, 1-Lipschitz function on $\Z_2$, with $N_1(f)=1.$ Then, $f$ is measure-preserving if and only if the van der Put coefficients $\{ b_{i}\}_{0 \leq i\leq3}$ satisfy the following conditions:
$$ b_0 +b_1 \equiv 1 \pmod{2};~b_2 \equiv 1 \pmod{2};~b_3\equiv 1 \pmod{2}.$$
\end{proposition}

\begin{proof}
It follows from Theorem \ref{MPvan} and (\ref{equivred}).
\end{proof}

By means of Proposition \ref{MPudmp2}, we are able to drop the assumption that $f$ is a 1-Lipschitz function on $\Z_2$ in the following result.

\begin{proposition}\label{MPudmp2sec}
A 1-Lipschitz function $f$ on $\Z_2$ is measure-preserving if and only if the following conditions are satisfied:

(i) $f$ is uniformly differentiable modulo $2$ with $N_1(f)=1.$

(ii)  $b_0 +b_1 \equiv 1 \pmod{2};~b_2 \equiv 1 \pmod{2};~b_3\equiv 1 \pmod{2}.$
\end{proposition}

\begin{proof}
The necessity follows from \cite [Proposition 9.24]{AK}. Indeed,
it suffices by Proposition \ref{MPudmp2} that  part (i) holds.
From \cite [Theorem 4.44]{AK}, $f$ is measure-preserving if and only if
$f =c+x +2g(x),$ where $c \in \Z_2$ is a constant and $g$ is a certain 1-Lipschitz function on $\Z_2$. Then part (i) follows directly by applying the explicit expression of $f$
to the definition of  a uniformly differentiable modulo 2 function.
For sufficiency, it is enough to show from Proposition \ref{MPudmp2} that $f$ is 1-Lipschitz if we use (i) and the second and third condition of (ii).
\end{proof}

In terms of the Mahler coefficients of  a 1-Lipschitz function $f$ represented by 
\begin{eqnarray}\label{Mex0}
f(x) &=&\sum_{m = 0}^{\infty} a_m\binom{x}{m}=\sum_{m = 0}^{\infty} p^{\lfloor {\rm log}_{p}m\rfloor}c_m\binom{x}{m}~~~(c_m \in \Z_p),
\end{eqnarray}
Proposition \ref{MPudmp2} has the following parallel result. 

\begin{proposition}
Let $f$ be a uniformly differentiable modulo $2$, 1-Lipschitz function on $\Z_2$, with $N_1(f)=1.$ Then, $f$ is measure-preserving if and only if the Mahler coefficient $a_1$ is a unit in $\Z_2$, that is,
$$ a_1 \equiv 1 \pmod{2}.$$
\end{proposition}

\begin{proof}
This follows from Proposition \ref{MPudmp2} by using the relations between the coefficients of Mahler and van der Put: noting that $c_0=a_0$ and  $c_1=a_1$,
\begin{equation}\label{relMV}
b_0 =c_0;~ b_1=c_0+c_1;~ b_2=c_1+c_2;~b_3=c_1 +3c_2+c_3.
\end{equation}

An alternative proof comes from a measure-preservation criterion in the Mahler expansions in \cite [Theorem 4.40]{AK} or \cite[Proposition 4.2]{J21b}: $a_1 \equiv 1 \pmod{2}$ and $a_n \equiv 0 \pmod{2^{s+1}}$ for all $n=n_{-}+2^s$ with $s\geq1.$
Noting that $c_2 \equiv 0 \equiv c_3 \pmod{2},$ Theorem \ref{Mcop2} makes the latter conditions redundant.
\end{proof}

Let us state a measure-preservation criterion for the case of $p\geq 3.$

\begin{proposition}\label{mpdiff}
Let $f$ be a uniformly differentiable modulo $p$, 1-Lipschitz function on $\Zp$, with $N_1(f)=1,$  whose  Mahler expansion is given by (\ref{Mex0}).
Then, $f$ is measure-preserving if and only if the following conditions are satisfied:
\begin{itemize}
\item[] {\rm (i)} $\{f(0), f(1), \cdots, f(p-1)\}$ is a complete set of all distinct residues modulo $p$;

\item[] {\rm (ii)} $b_{i+p} \not \equiv 0 \pmod{p}$ for all $0 \leq i\leq p-1;$

Equivalently, for all $0 \leq i<p$,
$$\sum_{n=0}^{p-1}\l_n^{i}c_n + \sum_{n=0}^{i}\binom{i}{n}c_{n+p}\not\equiv 0\pmod{p},$$
where $\l_n^{i}$ is an integer given by $\l_n^{i} = \frac{1}{p}(\binom{i+p}{n}-\binom{i}{n}).$
\end{itemize}
\end{proposition}
\begin{proof}
See \cite[Propostion 6.5]{J21b} for a direct proof.
Alternatively, this follows from Proposition \ref{MPvan} and the relation in (\ref{equivred}),  along with the formula for $B_{i+p}$ in (\ref{Bmform}) and (\ref{Amnfor}).
Indeed,

\begin{eqnarray*}
  B_{i+p}=\sum_{n=0}^{i+p}A_{i+p}(n)a_n  &=& \sum_{n=0}^{i+p}\left( \binom{i+p}{n}-\binom{i}{n} \right) a_n \\
  &=& \sum_{n=0}^{p-1}\left( \binom{i+p}{n}-\binom{i}{n} \right) c_n
+\sum_{n=p}^{i+p}\binom{i+p}{n} pc_n.
\end{eqnarray*}
On dividing by $p$, the Lucas congruence gives the result.

\end{proof}
We remark here that an alternative equivalent expression for $\l_n^i$  in \cite{Me20b} is given by
$$ p\l_n^i =\sum_{n+j \geq p, 1 \leq j \leq p-1} \binom{i}{n+j-p} \binom{p}{j}.$$

\begin{remark}
{\rm In contrast to the family of $B$-class functions on $\Z_p$ that are everywhere differentiable in \cite{An2},
uniformly differentiable modulo $p$ functions are not necessarily differentiable but satisfy  the same criterion for measure-preservation in \cite{J21}. }
\end{remark}

\begin{remark}
{\rm
Proposition  \ref{mpdiff} can be compared with Yurov's criterion in \cite[Theorem 1]{Yur98} for measure-preservation of a uniformly differentiable modulo $p$ function $f$ on $\Zp.$ Note that the function under consideration is not necessarily 1-Lipschitz, and that additionally, if $f$ is asymptotically 1-Lipschitz, i.e., a function satisfying the property that $a\equiv b \pmod{p^n}$ implies that  $f(a)\equiv f(b)\pmod{p^n}$ for any $n\geq N$ with some integer $N>0,$ then two notions of ``$f$ is measure-preserving '' and ``$f$ is asymptotically measure-preserving '' coincide, where the latter notion is described as the existence of a positive integer $B$ such that the reduced function $f_{/n}$ is bijective modulo $p^n$ for any $n\geq B.$ See Section 2 of \cite{Yur98} for details.}
\end{remark}

\subsection{Memi$\acute{c}$ et al.'s criterion for egrodicity revisited}

Memi$\acute{c}$ et al. \cite{MM} provided necessary and sufficient conditions for a uniformly differentiable modulo $p$, 1-Lipschitz function on $\Zp$ to be ergodic in terms of
the van der Put coefficients. We revisit their criterion for ergodicity and give an complete description of ergodicity for a uniformly differentiable modulo $p$, 1-Lipschitz function on $\Zp$ in terms of the coefficients of Mahler and van der Put.
We first recall the ergodicity criterion of Memi$\acute{c}$ et al. \cite{MM}, which can be rephrased in a slightly different form.

\begin{theorem}[see \cite{MM}]\label{Mcri}
Let $f$ be  a uniformly differentiable modulo $p$, 1-Lipschitz function on $\Zp$.
Then $f$ is ergodic if and only if the following conditions are satisfied:

\begin{enumerate}
    \item[(1)]  $f$ is transitive modulo $p;$

     \item[(2)]  $f$ is measure-preserving;

     \item[(3)]  For all integer $s>0$, $f^{p^s}(0) \not \equiv 0 \pmod{p^{s+1}};$

     \item[(4)]  For all integer $s>0$,
     $\frac {\prod_{j=0}^{p^s-1}B_{j+p^s}}{p^{sp^s}} \equiv 1 \pmod{p}.$

\end{enumerate}
\end{theorem}

To simplify Theorem \ref{Mcri}, we recall Anashin's ergodicity criterion for a uniformly differentiable modulo $p^2$, 1-Lipschitz function on $\Zp.$

\begin{theorem}(see \cite[Theorem 4.55]{AK})\label{Audmp2}
Let $f$ be a uniformly differentiable modulo $p^2$, 1-Lipschitz function on $\Zp.$
Then $f$ is ergodic if and only if
$f$ is transitive modulo $p^N$ for some positive integer $N$ or equivalently, for every $n\geq N,$
where $N=N(p)=N_2(f)+2$ for $p=2,$ $N=N_2(f)+1$ for $p$ an odd prime.
\end{theorem}

The following lemma in \cite{AK}  is crucial in proving Theorem \ref{Audmp2}.

\begin{lemma}(see \cite[Lemma 4.56]{AK}) \label{A-Lemma}
Let $f$ be a uniformly differentiable modulo $p$, 1-Lipschitz function on $\Zp$ and
let $f$ be transitive modulo $p^{s}$ for some $s\geq N_1(f)+1.$
Then $f$ induces either a transitive function modulo $p^{s+1}$ or a permutation that is a product of $p$ pairwise disjoint cycles of length $p^s$ each.
\end{lemma}

Anashin asked of an analogue of Theorem \ref{Audmp2} for uniformly differentiable modulo $p$, 1-Lipschitz functions on $\Zp$. After proving Lemma \ref {A-Lemma}, he made an important remark that
if $f$ is transitive modulo $p^{N_1(f)+1},$ then
the following congruence holds for every $x \in \Zp;$
\begin{equation}\label{A-remark}
  \prod_{j=0}^{p^{N_1(f)}-1}\partial_{1}(f)(f^j(x)) \equiv 1 \pmod{p}.
\end{equation}
Because of the property in (\ref{equsns3}),
the congruence in (\ref{A-remark})
is equivalent to the following congruence concerning the normalized van der Put coefficients :
\begin{equation}\label{eqcon4}
\prod_{j=0}^{p-1}b_{j+p}\equiv 1 \pmod{p},
\end{equation}
because $f$ is transitive modulo $p^{N_1(f)}.$
In fact, from (\ref{A-remark}) for $x=0,$
we have
$$ \prod_{j=0}^{p^{N_1(f)}-1}\partial_{1}(f^j(0))
\equiv
\prod_{j=0}^{p^{N_1(f)}-1}\partial_{1}(f)(j)
\equiv \left(\prod_{j=0}^{p-1}b_{j+p}\right)^{p^{N_1(f)-1}} \pmod{p}.$$

Let us make some observations on Theorem \ref{Mcri}.
A key idea for the proof of Theorem \ref{Mcri} is to prove that under the assumption of the transitivity of $f$ modulo $p^{s}$ for some integer $s>0$ or condition (3), $f$ is transitive modulo $p^{s+1}$ if and only if condition (4) holds for all integer $s>0$. Using the property that $B_{j+p^s}\equiv p^{s-1}B_{\bar{j}+p}$ from (\ref{ply2}), where $ \bar{j}$ is the reduction of $j$ modulo $p,$
it is crucial to see that
condition (4) of Theorem \ref{Mcri} is equivalent to the congruence in (\ref{eqcon4}).

Indeed, we calculate
$$(\prod_{j=0}^{p^s-1}B_{j+p^s})/p^{sp^{s}} \equiv (\prod_{j=0}^{p-1} p^{s}b_{j+p})^{p^{s}} /p^{sp^{s}}
=\left(\prod_{j=0}^{p-1} b_{j+p}\right)^{p^{s-1}} \pmod{p},$$
which leads to the equivalence of condition (4) of Theorem \ref{Mcri} and (\ref{eqcon4}).

It is important to observe that the congruence in (\ref{eqcon4}) does not depend on $s,$
and that from the criterion of Khrennikov  and  Yurova \cite{KY},  this congruence as well as condition (1) are equivalent to condition (2) of Theorem \ref{Mcri}, which can be redundant.
Furthermore, by Anashin's remark, condition (4) of Theorem \ref{Mcri} can be dropped if $f$ is transitive modulo $p^{N_1(f)+1}.$
In terms of these observations, it is natural to ask and determine the minimal level $s$ for which $f$ is transitive modulo $p^s$ so that $f$ is transitive for all higher levels, resulting in  ergodicity of $f.$

We invoke the following examples from \cite[Remark]{DP} of  uniformly differentiable modulo $p^2$, 1-Lipschitz functions with a special property.

\begin{example}\label{counter}
{\rm Let $f(x)=1+3x+2x^3 \in \Z_{2}[x].$
Then by Taylor's theorem, it is easy to show
that it is a uniformly differentiable modulo $p$, 1-Lipschitz function, with
$N_1(f)=1.$ A direct computation shows that
$f$ is transitive modulo $4$ whose a full cycle is (0,1,2,3) but not transitive modulo $8$.

Let $f(x)=1+4x+4x^3 +2x^5\in \Z_{3}[x].$
Then by Taylor's theorem, it is easy to show
that it is a uniformly differentiable modulo $p$, 1-Lipschitz function, with
$N_1(f)=1.$ It is shown by a direct computation that
$f$ is transitive modulo $9$ whose a full cycle is (0,1,2,6,7,5,3,4,8) but not transitive modulo $27$.
For a prime $p\geq 5,$ the polynomial $x^p+1$ can be shown to be  a uniformly differentiable modulo $p$, 1-Lipschitz function on $\Zp$ and transitive modulo $p$ but not modulo $p^2.$}
\end{example}

In light of Example \ref{counter} and Lemma \ref{A-Lemma}, it seems reasonable to take $\mu = \mu(p)= N_1(f) +2$ if  $p$ is 2 or 3 and $N_1(f) +1$ otherwise in Theorem \ref{mainR}.

The answer to the above question is an analogue of Theorem \ref{Audmp2} for uniformly differentiable modulo $p$, 1-Lipschitz functions on $\Zp.$

\begin{theorem}\label{mainR}
Let $f$ be a uniformly differentiable modulo $p$, 1-Lipschitz function, of $N_1(f)=1$ on $\Zp.$
Then $f$ is ergodic if and only if
 $f$ is transitive modulo $p^{\mu},$
  where $\mu:=\mu(p) =3$ for $p=2$ or 3 and $ \mu =2 $ for primes $p\geq 5.$
\end{theorem}

\begin{proof}
From the arguments of the proof of Theorem \ref{Mcri} and the equivalence of three congruences in condition (4) of Theorem \ref{Mcri}, (\ref{A-remark}), and (\ref{eqcon4}), it suffices to show that $f$ is transitive modulo $p^{\mu}$  for each case of $p.$ In the next subsections the proof for each $p$  will be provided.
\end{proof}

We close this subsection by pointing out that Theorem \ref{mainR} holds for polynomials over $\Zp$ in \cite[Proposition 9]{DP} or Proposition \ref{minred} and 1-Lipschitz functions of $B$-class in \cite{J21b}  and that in Lemma \ref{A-Lemma} the latter case of its conclusion does not occur if $N_1(f)=1.$

\subsection{Characterization of $p=2$}

In this section, we provide a complete description of ergodicity for a uniformly differentiable modulo $2$, 1-Lipschitz function, of $N_1(f)=1$  on $\Z_2$ in terms of the van der Put and Mahler   coefficients of $f$ whose expansion is  represented by
\begin{eqnarray}\label{Mexpis2}
f(x) &=&\sum_{m = 0}^{\infty} a_m\binom{x}{m}=\sum_{m = 0}^{\infty} 2^{\lfloor {\rm log}_{2}m\rfloor}c_m\binom{x}{m}~~~(c_m \in \Z_2).
\end{eqnarray}
From Theorem \ref{Mcop2}, it is evident that
the reduced function of $f$ modulo $2^3$ is equal to
$$f_{/3}(x):=\sum_{m = 0}^{3} a_m\binom{x}{m}=\sum_{m = 0}^{3} 2^{\lfloor {\rm log}_{2}m\rfloor}c_m\binom{x}{m},$$
which is a polynomial of degree at most 3 in  $\Z_2[x].$

For ergodicity of $f$ it suffices to show that
$f_{/3}$ is transitive modulo $2^3.$
By Larin's criterion \cite{La} for transitivity modulo 8 of  a polynomial
$g(x)=Ax^3+Bx^2 +Cx +D \in \Z_2[x],$ the polynomial coefficients are satisfied:
\begin{equation} \label{pergp2}
    A \equiv 0 \pmod{4},~ B \equiv 0 \pmod{2},~ B+C  \equiv 1 \pmod{4},~ D \equiv 0 \pmod{2}.
\end{equation}

From the relations between polynomial coefficients and Mahler's coefficients,
we obtain
$$ A=\frac{c_3}{3},~ B= c_2-c_3,~ C=c_1-c_2+\frac{2c_3}{3},~ D=c_0 $$
By substituting these relations into (\ref{pergp2}), we have a desired result for $\{c_i\}_{0\leq i \leq 3}$ in Theorem \ref{MErgp2} below.

\begin{theorem}\label{MErgp2}
Let $f: \Z_2 \rightarrow \Z_2$ be a uniformly differentiable modulo $2$, 1-Lipschitz function, of $N_1(f)=1$  on $\Z_2$ whose Mahler expansion is given by (\ref{Mexpis2}).
Then, $f$ is ergodic if and only if the following conditions are satisfied:
$$a_0 \equiv 1 \pmod{2},~ a_1 \equiv 1 \pmod{4},~ a_2 \equiv 0 \pmod{4},~ a_3 \equiv 0 \pmod{8}.$$
Equivalently,
$$c_0 \equiv 1 \pmod{2},~ c_1 \equiv 1 \pmod{4},~ c_2 \equiv 0 \pmod{2},~ c_3 \equiv 0 \pmod{4}.$$
\end{theorem}

The following well-known result is immediate from Theorem \ref{MErgp2}.
\begin{corollary}\label{listp2}
 Let $f$ be the same as above.  Then, $f$ is ergodic if and only if
 the map $x \mapsto f(x) \pmod{8}, x \in \{ 0, 1 \cdots ,7\}$ coincides with a map of the residue class ring
 $\Z /8\Z$ induced by any of the following 16 polynomials:
 $$
 \begin{array}{cccc}
                 x+1 & 5x+1 & 2x^2 +3x+1 &2x^2 +7x+1 \\
                 x+3 & 5x+3 & 2x^2 +3x+3 &2x^2 +7x+3 \\
                 x+5 & 5x+5 & 2x^2 +3x+5& 2x^2 +7x+5\\
                x+7 & 5x+7 & 2x^2 +3x+7 &2x^2 +7x+7
 \end{array}
 $$
\end{corollary}

\begin{proof}
See \cite{La} for the original proof.
\end{proof}

As in Mahler's expansion, we need to find a criterion for van der Put's coefficients  under the assumption that $f$ is transitive modulo $8.$
From the van der Put expansion of $f,$
it is easy to see that
the reduced function of $f$ modulo $2^3$ is equal to
$$f_{/3}:=\sum_{m =0}^{7} B_m \chi(m, x) =\sum_{m = 0}^{7} 2^{\lfloor {\rm log}_{2}m\rfloor}b_m\chi(m, x).$$
As $b_i \equiv 1 \pmod{2}$ for $4 \leq i \leq 7$ from the measure-preservation of $f$ in Proposition \ref{MPudmp2},
$$f_{/3}:=\sum_{m = 0}^{3} 2^{\lfloor {\rm log}_{2}m\rfloor}b_m\chi(m, x) +4\sum_{i=4}^{7}\chi(i, x) \pmod{8}.$$
From the relation in (\ref{relMV}), we obtain
$$ c_0=b_0,~c_1=-b_0+b_1,~c_2=b_0 -b_1 +b_2,~
b_3=-2b_0 +2b_0 -3b_2 +b_3.$$
Substituting these relations into the expressions for $\{c_i\}_{0\leq i\leq 3}$ of Theorem \ref{MErgp2} yields the following result.

\begin{theorem}\label{ErgBMv}
Let $f: \Z_2 \rightarrow \Z_2$ be a uniformly differentiable modulo $2$, 1-Lipschitz function, of $N_1(f)=1$  on $\Z_2$  whose van der Put expansion is given by (\ref{vanexp0}).
Then, $f$ is ergodic if and only if the following conditions are satisfied:
$$b_0 \equiv 1 \pmod{2},~ b_0 +b_1 \equiv 3 \pmod{4},~ b_2 \equiv 1 \pmod{2},~ b_2+ b_3 \equiv 2 \pmod{4}.$$
\end{theorem}

\begin{theorem}\label{ErgBoth}
Let $f: \Z_2 \rightarrow \Z_2$ be a 1-Lipschitz function on $\Z_2$ whose Mahler expansion and van der Put expansion is given by (\ref{Mexpis2}) and by (\ref{vanexp0}) respectively.
Then, $f$ is ergodic if and only if one of the following equivalent conditions is satisfied:

(i)
$$b_0 \equiv 1 \pmod{2},~ b_0 +b_1 \equiv 3 \pmod{4},~ b_2 \equiv 1 \pmod{2},~ b_2+ b_3 \equiv 2 \pmod{4}.$$

(ii) $$ c_0 \equiv 1 \pmod{2},~ c_1 \equiv 1 \pmod{4},~ c_2 \equiv 0 \pmod{2},~ c_3 \equiv 0 \pmod{4}.$$
\end{theorem}

\begin{proof}
It follows from Theorems \ref{ErgBMv} and \ref{MErgp2}  and Proposition \ref{MPudmp2sec}.
\end{proof}

\begin{remark}

{\rm  By virtue of Theorem \ref{ErgBoth}, we can exclude the following conditions from the ergodic conditions for 1-Lipschitz functions on $\Z_2$ from \cite[Theorem 4.40]{AK}:
For all $k\geq 2$,
\begin{equation}\label{Extra}
    \sum_{n=2^k}^{2^{k+1}-1} b_n \equiv 0 \pmod{4} \Leftrightarrow c_{2^{k+1}-1} \equiv 0 \pmod{4}.
\end{equation}
See \cite[Lemma 4.5]{J21b} for this equivalence. }
\end{remark}

\begin{remark}
{\rm  It is of great interest to show by a direct computation that each of two congruences in (\ref{Extra}) holds for a uniformly differentiable modulo $2$ function, with $N_1(f)=1$  on $\Z_2$ satisfying all conditions on the coefficients of Mahler and van der Put in Theorems \ref{ErgBMv} and \ref{MErgp2}, respectively.}
\end{remark}

\subsection{Characterization of $p=3$}
Here, we give  a complete description of ergodicity of a uniformly differentiable modulo $3^2$, 1-Lipschitz function $f$, of $N_1(f)=1$  on $\Z_3$ in terms of the Mahler and van der Put  coefficients.

From Theorem \ref{Mcop3over}, the reduced function  of $f$ modulo $3^3$ has the following Mahler expansion:

\begin{eqnarray}\label{f2p3}
f_{/3}(x) &=&\sum_{m =0}^{8} a_m\binom{x}{m}=\sum_{m =0}^{8} 3^{\lfloor {\rm log}_{3}m\rfloor}c_m\binom{x}{m}~~~(c_m\in \Z_3);
\end{eqnarray}
For ergodicity of $f,$ it suffices to show that $f$ is transitive modulo $3^3$, which is equivalent to the transitivity of $f_{/3}.$ As $f_{/3}$ is a polynomial of degree at most 8 in $\Z_3[x],$ we need to invoke a well-known result due to Jeong \cite{J21}.

To this end, we set the following quantities associated with a polynomial
$g(x)=\a_0+\a_1x +\cdots +\a_8x^8 \in \Z_3[x]$ of degree at most 8:
\begin{eqnarray}\label{constp3}
  A_1 &=& \sum_{i=1}^{4} \a_{2i-1}; ~ A_2:=\sum_{i=1}^{4} \a_{2i}; \\
  D_1 &=& \sum_{i=1}^{8} i\a_i; ~ D_2:=\sum_{i=1}^{8} i\a_i(-1)^{i-1}
\end{eqnarray}

\begin{proposition}\label{deg8}
A polynomial, $g(x)=\a_0+\a_1x +\cdots +\a_8x^8 \in \Z_3[x]$, is minimal if and only if  
$g$ fulfills one of the conditions, (i)--(viii):

(setting $[\a_0, A_1, A_2, \a_1,D_1,D_2]$ ${\rm mod}3=[\cdot, \cdot,\cdots, \cdot])$

\begin{itemize}
\item[] {\rm (i)} $[1,1,0,1,1,1],~ \a_2 + \a_4 +\a_6+ \a_8 +6 \not \equiv 0 ~[9],~ \a_2 +\a_4+\a_6+7\a_8+ 6  \not \equiv 0~[9];$

\item[] {\rm (ii)} $[1,1,0,1,2,2],~\a_0+ \a_1 + \a_3 +\a_5+ \a_7 +4   \not \equiv 0 ~[9],~  \a_0 + \a_1 + 6\a_2 +\a_3+ 7\a_5 + \a_7 + 4 \not \equiv0 ~[9];$

\item[] {\rm (iii)} $[1,1,0,2,1,2],~2\a_0+ \a_1 + \a_3 +\a_5+ \a_7 +3 \not \equiv 0 ~[9], 2\a_0 + \a_1 + 3\a_2+ \a_3 + 7\a_5 + \a_7 + 3   \not \equiv 0  ~[9];$

\item[] {\rm (iv)} $[1,1,0,2,2,1],~ 2\a_0 + \a_2 + \a_4 + \a_6 + \a_8 + 4  \not \equiv 0 ~[9],~ 2\a_0 + 4\a_2 + \a_4 + \a_6 + 7\a_8 + 4\not \equiv 0 ~[9];$

\item[] {\rm (v)} $[2,1,0,1,1,1],~ \a_2 + \a_4 + \a_6 + \a_8 + 3\not \equiv 0 ~[9],~ \a_2 + \a_4 + \a_6 + 7\a_8 + 3 \not \equiv 0~[9];$

\item[] {\rm (vi)} $[2,1,0,1,2,2],~2\a_0+\a_1 + \a_3 + \a_5 + \a_7 + 7  \not \equiv 0 ~[9],~ 2\a_0+\a_1 + 3\a_2+\a_3 + 7\a_5 + \a_7 +7 \not \equiv0 ~[9];$

\item[] {\rm (vii)} $[2,1,0,2,1,2],~2\a_0+\a_2 + \a_4 + \a_6 + \a_8 + 5\  \not \equiv 0 ~[9],2\a_0 + 4\a_2 + \a_4 + \a_6 + 7\a_8 +5\not \equiv 0  ~[9]; and$

\item[] {\rm (viii)} $[2,1,0,2,2,1],~  \a_0 + \a_1 + \a_3 + \a_5 + \a_7+6   \not \equiv 0 ~[9],~\a_0 + \a_1 + 6\a_2+\a_3 + 7\a_5 + \a_7 + 6\not \equiv 0 ~[9].$
\end{itemize}
\end{proposition}
\begin{proof}
See \cite[Corollary 5.4]{J21b}.
\end{proof}

We now use Proposition \ref{deg8} to establish the main result for $p=3.$
From Theorem \ref{Mcop3over}, we note that $c_{i}'=\frac{1}{3}c_i$ for $6 \leq i \leq 8.$ By solving for $\{\a_i\}_{0 \leq i \leq 8}$ in the matrix equation between the polynomial coefficients $\{\a_i\}_{0 \leq i \leq 8}$ of $g$ and the Mahler coefficients $\{c_i\}_{0 \leq i \leq 8}$ of $f_{/3}$ we obtain the following relation:

\begin{gather}\label{relpis3}
 \begin{bmatrix}
      \a_0 \\
      \a_1 \\
      \a_2 \\
      \a_3\\
      \a_4 \\
      \a_5 \\
      \a_6 \\
      \a_7\\
      \a_8
 \end{bmatrix}
 \equiv
  \begin{bmatrix}
    1 & 0 & 0 & 0 & 0 & 0 & 0 & 0 & 0 \\
    0 & 1 & 4 & 1 & 6 & 6 & 3 & 0 & 0 \\
    0 & 0 & 5 & 3 & 7 & 1 & 5 & 0 & 0 \\
    0 & 0 & 0 & 5 & 6 & 2 & 0 & 2 & 5 \\
    0 & 0 & 0 & 0 & 8 & 2 & 5 & 6 & 4 \\
    0 & 0 & 0 & 0 & 0 & 7 & 6 & 2 & 8 \\
    0 & 0 & 0 & 0 & 0 & 0 & 8 & 3 & 1 \\
    0 & 0 & 0 & 0 & 0 & 0 & 0 & 5 & 5 \\
    0 & 0 & 0 & 0 & 0 & 0 & 0 & 0 & 4
  \end{bmatrix}
    \begin{bmatrix}
      c_0 \\
      c_1 \\
      c_2 \\
      c_3\\
      c_4 \\
      c_5 \\
      c_6' \\
      c_7'\\
      c_8'
    \end{bmatrix}
    \pmod{9}.
\end{gather}

The relations in (\ref{constp3}) give an explicit expression between two vectors in Proposition \ref{deg8} and Theorem \ref{ErgBM}:

\begin{gather}\label{matrelpis3}
 \begin{bmatrix}
       c_0 \\
       c_1 \\
       c_2 \\
       c_3\\
       c_4 \\
       c_5
 \end{bmatrix}
 \equiv
  \begin{bmatrix}
    1 & 0 & 0 & 0 & 0 & 0\\
    0 & 1 & 1 & 0 & 0 & 0  \\
    0 & 0 & 2 & 0 & 0 & 0  \\
    0 & 1 & 1 & 1 & 0 & 0  \\
    0 & 1 & 2 & 1 & 1 & 0 \\
    0 & 1 & 0 & 1 & 2 & 1
  \end{bmatrix}
    \begin{bmatrix}
      a_0 \\
      A_1 \\
      A_2 \\
      a_1\\
      D_1 \\
      D_2
    \end{bmatrix}
    \pmod{3}.
\end{gather}

By direct computations using (\ref{relpis3}) and (\ref{matrelpis3}),  Proposition \ref{deg8} gives the following parallel result.
\begin{theorem}\label{ErgBM}
Let $f$ be a uniformly differentiable modulo $3$, 1-Lipschitz function, of $N_1(f)=1$  on $\Z_3$ whose Mahler expansion is given by (\ref{f2p3}).
Then, $f$ is ergodic if and only if  $f$ fulfills one of the conditions (i)--(viii):

Setting $[c_0, c_1,c_2,c_3, c_4,c_5]~{\rm mod}~3=[\cdot, \cdot,\cdots, \cdot]_m,$

\begin{itemize}
\item[] {\rm (i)} $[1,1,0,0,0,0]_m,~ c_2 +3 \not \equiv 0 ~[9],~ c_2 +c_8+3 \not \equiv 0~[9];$

\item[] {\rm (ii)}  $[1,1,0,0,1,2]_m,~c_0+ c_1 +c_2 +1  \not \equiv 0 ~[9], c_0+c_1 +c_2 + c_6+ c_7+ c_8+ 4 \not \equiv 0  ~[9];$

\item[] {\rm (iii)}  $[1,1,0,1,2,2]_m,~  2c_0+c_1 +c_2   \not \equiv 0 ~[9],~ 2c_0 +c_1 +c_2 +2{c_6}+{c_7}+c_8+6\not \equiv 0 ~[9];$

\item[] {\rm (iv)} $[1,1,0,1,0,2]_m,~c_0+c_2 +2   \not \equiv 0 ~[9],~  c_0 + c_2 +c_6 +c_8 +5 \not \equiv0 ~[9];$

\item[] {\rm (v)} $[2,1,0,0,0,0]_m,~ c_2 +6 \not \equiv 0 ~[9],~ c_2 +c_8+6\not \equiv 0~[9];$

\item[] {\rm (vi)}  $[2,1,0,0,1,2]_m,~c_0+ 2c_1 +2c_2 +5  \not \equiv 0 ~[9], c_0 +2c_1 +2c_2+ c_6+ 2c_7+ 2c_8+2 \not \equiv 0  ~[9];$

\item[] {\rm (vii)} $[2,1,0,1,2,2]_m,~   c_0+ c_2 +4    \not \equiv 0 ~[9],~ c_0 +c_2 +{c_6}+{c_8}+1\not \equiv 0 ~[9];$ and

\item[] {\rm (viii)} $[2,1,0,1,0,2]_m,~2c_0 +2c_1+ 2c_2 +3 \not \equiv 0 ~[9],~ c_0+c_1 +{c_2}+{c_6}+c_7+{c_8} +3 \not \equiv0 ~[9].$
\end{itemize}
\end{theorem}

Parallel to Theorem \ref{ErgBM} we have the following result in the van der Put expansion represented by
\begin{eqnarray}
f(x) &=&\sum_{m =0}^{\infty} B_m \chi(m, x) =\sum_{m =0}^{\infty} 3^{\lfloor {\rm log}_{3}m\rfloor}b_m\chi(m, x)~~~(b_m \in \Z_3). \label{vanexp}
\end{eqnarray}

\begin{theorem}\label{ErgBM2}
Let $f$ be a uniformly differentiable modulo $3$, 1-Lipschitz function, of $N_1(f)=1$  on $\Z_3$  whose  van der Put  expansion is given by (\ref{vanexp}).
Then, $f$ is ergodic if and only if  
$f$ fulfills one of the conditions (i)--(viii):

Setting $[b_0, b_1, b_2,b_3, b_4, b_5]~ {\rm mod}~3 =[\cdot, \cdot,\cdots, \cdot]_v,$
\begin{itemize}
\item[] {\rm (i)} $[1,2,0, 1, 1, 1]_v$, $b_0 + b_1 +b_2 \not \equiv 3~[9],$ $\sum_{m=0}^{8} b_m \not \equiv 3 ~[9]$;

\item[] {\rm (ii)} $ [1,2,0, 1, 2, 2]_v$, $b_0 +2 b_1 +b_2 \not \equiv 5~[9],$ $\sum_{m=0}^{8} b_m \not \equiv 2(b_1 +b_4 +b_7+ 1) ~[9];$

\item[] {\rm (iii)} $[1,2,0, 2, 1, 2]_v $, $2b_0 + 2b_1 +b_2 \not \equiv 6~[9],$ $\sum_{m=0}^{8} b_m \not \equiv 2(b_2 +b_5 +b_8 +3)~[9];$

\item[] {\rm (iv)} $[1,2,0, 2, 2, 1]_v $, $2b_0 + b_1 +b_2 \not \equiv 4 ~[9],$ $\sum_{m=0}^{8} b_m \not \equiv 2(b_0 +b_3 +b_6 +2)~[9];$

\item[] {\rm (v)} $[ 2,0,1, 1, 1, 1]_v$, $b_0 + b_1 +b_2 \not \equiv 3~[9],$ $\sum_{m=0}^{8} b_m \not \equiv 3 ~[9]$;

\item[] {\rm (vi)} $ [2,0,1, 1, 2, 2]_v$, $b_0 +b_1 +2b_2 \not \equiv 4~[9],$ $\sum_{m=0}^{8} b_m \not \equiv 2(b_2 +b_5 +b_8 +2) ~[9];$

\item[] {\rm (vii)} $[2,0,1, 2, 1, 2]_v $, $2b_0 + b_1 + b_2\not \equiv 5~[9],$ $\sum_{m=0}^{8} b_m \not \equiv 2(b_0 +b_3 +b_6 +1) ~[9];$ and

\item[] {\rm (viii)} $[2,0,1, 2, 2, 1]_v $, $ 2b_0 + b_1 + 2b_2  \not \equiv 6~[9],$ $\sum_{m=0}^{8} b_m \not \equiv 2(b_1 +b_4 +b_7+ 3)~[9]$.
\end{itemize}
\end{theorem}

\begin{proof}
The proof follows immediately from Theorem \ref{ErgBM} by transporting it into the equivalent result in terms of van der Put's coefficients.
In this transition, it is crucial to deduce the following relation between the coefficients  of Mahler and van der Put  from (\ref{Mcof}):
\begin{gather*}\label{mmp}
 \begin{bmatrix}
      c_0 \\
      c_1 \\
      c_2 \\
      c_3\\
      c_4 \\
      c_5 \\
      c_6 \\
      c_7\\
      c_8
 \end{bmatrix}
 =
  \begin{bmatrix}
    1 & 0 & 0 & 0 & 0 & 0 & 0 & 0 & 0 \\
   -1 & 1 &  0& 0 & 0 & 0 & 0 & 0 & 0 \\
    1 & -2 & 1 & 0 & 0 & 0 & 0 & 0 & 0\\
    0 & 1 & -1& 1 & 0 & 0 & 0 & 0 & 0 \\
    -1 & -1 & 2 & -4 & 1 & 0 & 0 & 0 & 0 \\
    3 & 0& -3 & 10 & -5 & 1 & 0 & 0 & 0 \\
    -6 & 3 & 3 & -20 & 15 & -6 & 1 & 0 & 0 \\
    9 & -9 & 0 & 35 & -35 & 21 & -7& 1 & 0 \\
    -9 &  18&  -9 & -56 &  70&  -56& 28 & -8 & 1
  \end{bmatrix}
    \begin{bmatrix}
      b_0 \\
      b_1 \\
      b_2 \\
      b_3\\
      b_4 \\
      b_5 \\
      b_6 \\
      b_7 \\
      b_8
    \end{bmatrix}
     .
\end{gather*}

The observation from (\ref{ply2}) that
$b_6\equiv 2b_3\pmod{3}$, $b_7\equiv 2b_4\pmod{3}$ and $b_8\equiv 2b_5\pmod{3}$ is also crucially used. The tedious verification is omitted.
\end{proof}


\subsection{Characterization of $p\geq 5$}
Here, for  a prime $p\geq 5,$  we use the van der Put coefficients to  give  a complete description of ergodicity of a uniformly differentiable modulo $p$, 1-Lipschitz function
$f$, of $N_1(f)=1$  on $\Z_p$ in Mahler's expansion represented by
\begin{eqnarray}
f(x) &=&\sum_{m =0}^{\infty} a_m\binom{x}{m}=\sum_{m =0}^{\infty} p^{\lfloor {\rm log}_{p}m\rfloor}c_m\binom{x}{m}~~~(c_m\in \Z_p); \label{Mexpp}
\end{eqnarray}
From Theorem \ref{Mcop3over}, it follows that

\begin{eqnarray}\label{f2pp}
f_{/2}(x) &=&\sum_{m =0}^{2p-1} a_m\binom{x}{m}=\sum_{m =0}^{2p-1} p^{\lfloor {\rm log}_{p}m\rfloor}c_m\binom{x}{m}~~~(c_m\in \Z_p);
\end{eqnarray}

For ergodicity of $f,$ we claim the transitivity of $f$ modulo $p^2$, being equivalent to transitivity of $f_{/2}.$ Since $f_{/2}$ is a polynomial of degree at most $2p-1$ on $\Zp,$ we verify the claim in two ways; one is to use Mahler's coefficients and the other is to use a polynomial as in $p=2$ or $p=3.$
The two approaches are both based on the following result.

\begin{proposition}\label{ergp5}
Let $p$ be a prime $\geq5$ and  $g$ be a polynomial in $\Zp[x]$
Then, $g$ is ergodic if and only if the following conditions are satisfied:

(i) $g$ is transitive modulo $p$;

(ii)  $(g^p)'(0)\equiv 1 \pmod{p}$, i.e., $b_{p}\cdots b_{2p-1}\equiv 1 \pmod{p};$ and

(iii) $g^p(0) \in p\Z_p\setminus p^2\Z_p$.

\end{proposition}

\begin{proof}
See \cite[Lemma 8 and Propostion 9]{DP} for a proof.
\end{proof}

Now, we use Proposition \ref{ergp5} to explain a procedure of finding conditions on Mahler's coefficients $\{ c_i \}_{0 \leq i \leq 2p-1}$ for which $f_{/2}$ is transitive.
Given a transitive polynomial $f_{/2}$ in (\ref{f2pp}), then
it is obvious that $f$ is transitive modulo $p,$ so there is a permutation $\varphi$ on the finite field $\F_p$ of full cycle induced by $f$.

\indent $\bullet$ Step 1: Find the Mahler coefficients $\{ c_m{\rm mod}p\}_{0\leq m \leq p-1}$  of the reduced function $f_{/2}$ in (\ref{f2pp})  such that for all $a=0, \cdots, p-1,$
$$f(a)\equiv f_{/2}(a) \equiv  \varphi(a) \pmod{p}.$$
One can find $\{c_m {\rm mod}p\}_{0\leq m \leq p-1}$ directly by solving the linear system above in which the coefficient matrix is a lower triangular matrix consisting of the Pascal triangle numbers. We also note that there are exactly $(p-1)!$ choices for $\{c_m {\rm mod}p\}_{0\leq m \leq p-1}$ corresponding to all possible choices for transitive permutations on $\F_p.$

\indent $\bullet$ Step 2: For any chosen constant vector $[b_{p}, \cdots, b_{2p-1}]$ satisfying  the congruence $b_{p}\cdots b_{2p-1}\equiv 1 \pmod{p},$ in condition (ii) of Proposition \ref{ergp5},  find the Mahler coefficients $\{c_m {\rm mod}p\}_{p\leq m \leq 2p-1}$ satisfying the linear system in Proposition \ref{mpdiff}:
for all $0 \leq i<p$
\begin{eqnarray} \label{lsp5}
\sum_{n=0}^{p-1}\l_n^{i}c_n + \sum_{n=0}^{i}\binom{i}{n}c_{n+p}\equiv b_{i+p}\pmod{p},
\end{eqnarray}
where $\l_n^{i} = \frac{1}{p}(\binom{i+p}{n}-\binom{i}{n}).$ Indeed, the linear system has a unique solution modulo $p$ because the Pascal triangle numbers form the lower triangular part of the coefficient matrix; furthermore, note that there are exactly $(p-1)^{p-1}$ choices for the vector $[b_{p},\cdots ,b_{2p-1}]$ satisfying condition (ii) of Proposition \ref{ergp5}.

\indent $\bullet$ Step 3: For any given elements $\{ c_m {\rm mod}p\}_{0\leq m \leq 2p-1},$ find the Mahler coefficients modulo $p^2$, $\{c_m {\rm mod}p^2\}_{0\leq m \leq p-1}$ that satisfy $f^p(0) \not \equiv 0 \pmod{p^2}$ in condition (iii) of Proposition \ref{ergp5}.
In this step, induction on $i$ shows that
$$ f^{i}(0)\equiv R_i(c_0,\cdots,c_{p-1})\pmod{p^2},$$ where
$ R_i(c_0,\cdots,c_{p-1})$ is a linear polynomial with coefficients in $ \mathbb{Z}/p^2 \mathbb{Z}.$
Indeed, by expressing $c_m ={e}_m+pz_m$ with $ 0 \leq {e}_m \leq p -1$ for each $0 \leq m\leq 2p-1$, the polynomial $f_{/2}(x)=\sum_{m=0}^{2p-1}c_m p^{\lfloor {\rm log}_{p}m\rfloor}\binom{x}{m}$ can be reduced to
\begin{eqnarray} \label{rform}
f_{/2}(x) \equiv \sum_{m=0}^{p-1}(e_m +pz_m)\binom{x}{m} +\sum_{m=0}^{p-1}p{z}_{m+p}\binom{x}{m+p}\pmod{p^2}.
\end{eqnarray}
With this representation in mind, it is straightforward to derive a desired form of  $R_i~(0\leq i\leq p)$ by applying Lucas's congruence  to the induction step. By virtue of Proposition \ref{ergp5}, $f$ is transitive modulo $p^2$, so it is ergodic. Indeed, the erogodic polynomial obtained this way turns out to be a polynomial with coefficients in $ \mathbb{Z}/p^2 \mathbb{Z}.$

\begin{remark}
{\rm The procedure described above can be compared with that proposed by Anashin
\cite[Page 296]{AK}. In contrast to  Anashin's approach, the proposed method  need not to  calculate the interpolation polynomials $f_{\varphi}$ and $f_{\varphi, \psi}$ in \cite{AK} and to test whether they are transitive modulo $p^2$. In this respect, this method is more practical and efficient than Anashin's.}
\end{remark}

We give another method of finding ergodic conditions on $\{ \a_i \}_{0 \leq i \leq 2p-1}$ for a polynomial
\begin{equation}\label{poldeg2p}
g(x)= \a_0 + \a_1x + \cdots +\a_{2p-1}x^{2p-1} \in  \Z/p^2\Z[x]
\end{equation}
of degree at most $2p-1$,
 associated with $f_{/2}.$
We set the constants, $B_0,\cdots B_{p-1}, D_0,\cdots D_{p-1}$, as follows:
\begin{eqnarray}
  \a_0 & = & B_0; \nonumber \\
  \a_1 + \a_p +\a_{2p-1} &=& B_1;\nonumber\\
  \a_2 +\a_{p+1} &=& B_2;\nonumber\\
  \vdots &=& \vdots ~ \nonumber\\
  \a_{p-1} +\a_{2p-2} &=& B_{p-1};\nonumber\\
  \a_1 &=& D_0;\label{praw}\\
  \sum_{i=1}^{2p-1} i\a_i&=& D_1;\nonumber\\
  \vdots &=& \vdots ~ \nonumber\\
 \sum_{i=1}^{2p-1} i(p-1)^{i-1}\a_i&=& D_{p-1},\nonumber
\end{eqnarray}

Note that because $x^p \equiv x \pmod{p}$, the polynomial $g$ is reduced modulo $p$ to
\begin{equation}\label{redf}
g(x)\equiv B_0 +B_1x +\cdots + B_{p-1}x^{p-1}.
\end{equation}
Furthermore, $g'(i)=D_i$ for each $ 0 \leq i \leq p-1.$ Next, we consider (\ref{praw}) as a linear system in variables, $ {\bf x} =[a_0, \cdots, a_{2p-1}]^t$, for a given constant column vector, ${\bf b} =[B_0, \cdots, B_{p-1}, D_{0}, \cdots, D_{p-1}]^t$ modulo $p,$ satisfying conditions (i) and (ii) of Proposition \ref{ergp5}:
\begin{equation}\label{Mxb}
M  {\bf x} \equiv {\bf b} \pmod{p},
\end{equation}
where $M$ is a $2p \times 2p$ coefficient matrix explicitly given by the following form:
$$M=$$
\begin{equation}\label{Matp}
\left( \begin{array}{ccccccccc}
1 & 0 & 0   & \cdots & 0 & 0  & \cdots & 0\\
0 & 1 & 0   & \cdots & 0 & 1  & \cdots & 1\\
0 & 0 & 1   & \cdots & 0 & 0  & \cdots & 0 \\
0 & 0 & 0   & \cdots & 0 & 0  & \cdots & 0\\
\vdots& \vdots  & \vdots &\vdots& \vdots  & \vdots& \vdots& \vdots\\
0 & 0 & 0   &\cdots& 1 & 0   & \cdots& 0\\
0 & 1 & 0   &\cdots& 0 & 0  & \cdots& 0\\
0 & 1& 2 &  \cdots & p-1 & p  & \cdots &  2p-1 \\
0 & 1 & 2\cdot2    &\cdots & (p-1)\cdot2^{p-2} & p\cdot2^{p-1}   & \cdots &(2p-1)\cdot2^{2p-2}\\
\vdots& \vdots  & \vdots &\vdots& \vdots  & \vdots& \vdots& \vdots\\
0 & 1 & 2\cdot(p-1) &\cdots & (p-1)\cdot(p-1) ^{p-2} & p\cdot(p-1) ^{p-1}  & \cdots &(2p-1)\cdot(p-1) ^{2p-2}
\end{array} \right)
\end{equation}

We now characterize the transitivity of $g$ modulo $p^2$ in terms of its coefficients.
\begin{theorem}\label{mainp5}

1. A polynomial $g$  of degree $\leq 2p-1$ in (\ref{poldeg2p}) is transitive modulo $p^2$, equivalently $g$ is ergodic on $\Zp$ if and only if
$$g(x) =g_0(x) +pg_1(x),$$ where

(i) $g_0(x)=\sum_{i=0}^{2p-1}{e_ix^i} \in \Z/p\Z[x]$ is a transitive polynomial modulo $p$ of degree at most $2p-1,$ whose coefficient column vector, $[e_0, \cdots, e_{2p-1}]^t$, is a unique solution to the linear system of the form, $${M}{\bf x} \equiv {\bf b} \pmod{p},$$ where
the coefficient matrix ${M}$ is given by (\ref{Matp}) and ${\bf b}=[B_0, \cdots, B_{p-1}, D_{0}, \cdots, D_{p-1}]^t$ in (\ref{praw}) is a given constant vector that is chosen among $(p-1)!(p-1)^{p-1}$ choices satisfying conditions (i) and (ii) of Proposition \ref{ergp5}.

(ii) the coefficient row vector, $[z_0, \cdots, z_{2p-1}]$ of $g_1(x)=\sum_{i=0}^{2p-1}{z_ix^i} \in \Z/p\Z[x]$,
satisfies the nonvanishing modulo $p$ of the linear polynomial $l$: $$ l(z_0,\cdots, z_{2p-1}) \not \equiv 0\pmod{p},$$ where $ l(z_0,\cdots, z_{2p-1})$ is given explicitly by the formula:
\begin{eqnarray}\label{lmin}
l(z_0,\cdots, z_{2p-1}) = \frac{1}{p} f_0^p(0)+ \frac{1}{g_0'(0)} z_0 + \sum_{i=1}^{p-1}w_i g_1(g_0^i(0)),
\end{eqnarray}
where, for $1 \leq i \leq p-2,$
\begin{eqnarray*}\label{wi}
w_i = \prod_{j=i+1}^{p-1}g_0'(g_0^j(0))~{\rm and }~w_{p-1}=1.
\end{eqnarray*}
\end{theorem}

\begin{proof}
  See \cite[Theorem 6.2] {J21b}.
\end{proof}

\begin{remark}
{\rm For a fixed constant column vector, ${\bf b}=[B_0, \cdots, B_{p-1}, D_{0}, \cdots, D_{p-1}]^t$ modulo $p$ out of $(p-1)!(p-1)^{p-1}$ choices, we can count the set $E$ of nonequivalent minimal polynomials of degree $\leq 2p-1$ with  coefficients in $\Z/p^2\Z$ that correspond to all coefficient vectors, $[z_0, \cdots z_{2p-1}]{\rm mod}p$, satisfying the condition, $l(z_0,\cdots, z_{2p-1}) \not \equiv 0 \pmod{p}$ in (\ref{lmin}). 
It is shown in \cite{J21b} that the set $E$  has cardinality of $(p-1)!(p-1)^{p} p^{p-1}$ by considering  an 
equivalence relation on the set $S$ defined by $$ S:=\{ [z_0, \cdots z_{2p-1}] \in (\Z/p\Z)^{2p}~|~ l(z_0,\cdots, z_{2p-1}) \not \equiv 0 \pmod{p} \}.$$ 
}
\end{remark}

\end{document}